\documentclass[leqno,10pt]{amsart}
\usepackage{amssymb,amsmath,amsfonts,amsthm,mathtools,latexsym,ifthen,bezier}
\usepackage{graphicx}
\usepackage{undertilde}
\usepackage{amsxtra}
\usepackage{xspace}
\usepackage{url}
\usepackage[english]{babel} 

\usepackage{enumitem}
\usepackage[margin=1in]{geometry}
\usepackage{etoolbox}

\newcommand{\kla}[1]{ {\langle #1 \rangle} }

\newcommand{\st}{\;|\;}

\newcommand{\dom}{ {\rm dom} }
\newcommand{\ran}{ {\rm ran} }
\newcommand{\otp}{ {\rm otp} }

\newcommand{\sub}{\subseteq}

\newfont{\ssi}{cmssi12 at 12pt}

\newcommand{\rest}{{\restriction}}
\newcommand{\cf}{ {\rm cf} }
\newcommand{\On}{ {\rm On} }

\newcommand{\leer}{\emptyset}
\newcommand{\ohne}{\setminus}

\newcommand{\id}{ {\rm id} }

\newenvironment{ea*}{\begin{eqnarray*}}{\end{eqnarray*}}

\newcounter{claimnumber}
\AtEndEnvironment{proof}{\setcounter{claimnumber}{0}}
{\begin{enumerate}[label=(\arabic*)]
\refstepcounter{claimnumber}
\setcounter{enumi}{\value{claimnumber}-1}
\item}%
{\end{enumerate}}

\setbox0=\hbox{$\longrightarrow$}
\newcommand{\To}{\longrightarrow}

\newcommand{\power}{{\mathcal{P}}}

\newcommand{\calC}{\mathcal{C}}

\newcommand{\bG}{{\bar{G}}}

\newcommand{\bM}{{\bar{M}}}
\newcommand{\bP}{{\bar{\P}}}
\newcommand{\bR}{{\bar{R}}}

\newcommand{\bX}{{\bar{X}}}

\newcommand{\bkappa}{{\bar{\kappa}}}

\newcommand{\bmu}{{\bar{\mu}}}

\newcommand{\btheta}{{\bar{\theta}}}

\newcommand{\barf}{{\bar{f}}}

\newcommand{\bp}{{\bar{p}}}

\newcommand{\bN}{{\bar{N}}}
\newcommand{\talpha}{{\tilde{\alpha}}}

\newcommand{\tlambda}{{\tilde{\lambda}}}

\newcommand{\tzeta}{{\tilde{\zeta}}}

\newcommand{\tN}{{\tilde{N}}}

\newcommand{\va}{{\vec{a}}}

\newcommand{\vx}{{\vec{x}}}

\newcommand{\vA}{{\vec{A}}}

\newcommand{\vT}{{\vec{T}}}

\newcommand{\vxi}{{\vec{\xi}}}

\newcommand{\seq}[2]{{\langle#1\;|\;}\linebreak[0]{#2\rangle}}

\renewcommand{\phi}{\varphi}
\newcommand{\card}[1]{\overline{\overline{#1}}}

\newcommand{\ZFC}{\ensuremath{\mathsf{ZFC}}\xspace}

\newcommand{\ZFCm}{\ensuremath{{\ZFC}^-}\xspace}

\newcommand{\AD}{\ensuremath{\mathsf{AD}}}

\newcommand{\V}{\ensuremath{\mathrm{V}}\xspace}

\newcommand{\forces}{\Vdash}

\def\<#1>{\langle#1\rangle}
\newcommand{\restrict}{\upharpoonright}
\newcommand{\B}{{\mathord{\mathbb{B}}}}
\renewcommand{\P}{{\mathord{\mathbb P}}}
\newcommand{\Q}{{\mathord{\mathbb Q}}}

\newcommand{\compose}{\circ}

\newcommand{\SC}{\ensuremath{\mathsf{SC}}\xspace}

\newcommand{\BFA}{\ensuremath{\mathsf{BFA}}}
\newcommand{\BPFA}{\ensuremath{\mathsf{BPFA}}\xspace}

\newcommand{\MP}{\ensuremath{\mathsf{MP}}}

\newcommand{\ColNothing}{\mathrm{Col}}
\newcommand{\Col}[1]{\ColNothing(#1)}

\newcommand{\MPColNothing}[1]{\MP_{\Col{\dot{\kappa}}}}

\newcommand{\CH}{\ensuremath{\mathsf{CH}}\xspace}

\newcommand{\naturals}{{\mathord{\mathbb{N}}}}

\newcommand{\reals}{{\mathord{\mathbb{R}}}}

\newcommand{\Hull}{\mathrm{Hull}} 

\newcommand{\code}[1]{\ulcorner{#1}\urcorner}

\newcommand{\SFP}[1]{\ensuremath{\mathsf{SFP}_{#1}}}

\newcommand{\DSR}[1]{\ensuremath{\mathsf{DSR}{(#1)}}}

\newcommand{\CC}{{\ensuremath{\text{$\sigma$-closed}}}}

\newcommand{\BSCFA}{\ensuremath{\mathsf{BSCFA}}\xspace}

\newcommand{\Goedel}{G\"{o}del\xspace}

\newcommand{\Prikry}{P\v{r}\'{\i}kr\'{y}\xspace}

\newtheorem{thm}{Theorem}[section]
\newtheorem*{thm*}{Theorem} 
\newtheorem{cor}[thm]{Corollary}
\newtheorem{lem}[thm]{Lemma}
\newtheorem{obs}[thm]{Observation}

\newtheorem{fact}[thm]{Fact}

\theoremstyle{definition}
\newtheorem{defn}[thm]{Definition}
\newtheorem{question}[thm]{Question}

\newcommand{\thistheoremname}{}
\newtheorem{genericthm}[thm]{\thistheoremname}

\theoremstyle{remark}

\newcommand{\sATP}[2]{\text{strong \ensuremath{\mathsf{ATP}_{#2}(#1)}}\xspace}
\newcommand{\wATP}[2]{\ensuremath{\mathsf{ATP}_{#2}(#1)}\xspace}
\newcommand{\SCabove}[1]{\ensuremath{\SC{\restrict}#1}\xspace}

\begin{document}

\title{Aronszajn tree preservation and bounded forcing axioms}


\author[Fuchs]{Gunter Fuchs}
\address[G.~Fuchs]{Mathematics,
          The Graduate Center of The City University of New York,
          365 Fifth Avenue, New York, NY 10016
          \&
          Mathematics,
          College of Staten Island of CUNY,
          Staten Island, NY 10314}
\email{Gunter.Fuchs@csi.cuny.edu}
\urladdr{http://www.math.csi.cuny/edu/$\sim$fuchs}

\thanks{The research of the author has been supported in part by PSC CUNY research grant 61567-00 49 and Simons Collaboration grant 580600. The author also wishes to express his gratitude to the logic group in Bonn for their hospitality in January 2019.}

\date{\today}     					

\begin{abstract}
I investigate the relationships between three hierarchies of reflection principles for a forcing class $\Gamma$: the hierarchy of bounded forcing axioms,  of $\Sigma^1_1$-absoluteness and of Aronszajn tree preservation principles. The latter principle at level $\kappa$ says that whenever $T$ is a tree of height $\omega_1$ and width $\kappa$ that does not have a branch of order type $\omega_1$, and whenever $\P$ is a forcing notion in $\Gamma$, then it is not the case that $\P$ forces that $T$ has such a branch. $\Sigma^1_1$-absoluteness serves as an intermediary between these principles and the bounded forcing axioms.  A special case of the main result is that for forcing classes that don't add reals, the three principles at level $2^\omega$ are equivalent. Special attention is paid to certain subclasses of subcomplete forcing, since these are natural forcing classes that don't add reals.
\end{abstract}
\maketitle

\section{Introduction}
One of the main observations in Fuchs \& Minden \cite{FuchsMinden:SCforcingTreesGenAbs} was that assuming the continuum hypothesis, the bounded forcing axiom for any natural%
\footnote{See Definition \ref{defn:ForcingEquivalenceAndNaturalness} for the meaning of ``natural.''}
class $\Gamma$ of forcing notions that don't add reals is equivalent to the statement that forcing notions in $\Gamma$ cannot add a cofinal branch to any tree of height and width $\omega_1$ that does not have a cofinal branch already (I call this latter property strong $(\omega_1,\omega_1)$-Aronszajn tree preservation).
This characterization was apparently mostly overlooked, as far as I can tell, maybe because the most well-known forcing classes whose forcing axioms have been widely considered in the literature may add reals. Even though the arguments establishing this characterization build on ``folklore'' results, it is worth carrying them out with care, because I did find a false statement on a claimed equivalence between bounded forcing axioms and $(\omega_1,\omega_1)$-Aronszajn tree preservation in the literature. Thus, in the introduction of Zapletal \cite{Zapletal:BNFAmayfail}, it is claimed that ``under the Continuum Hypothesis, it [the bounded forcing axiom for a forcing notion $\P$] is equivalent to the statement that $\P$ does not add any branches to trees of height and width $\omega_1$ which have no branches in the ground model''. No reference is given for this, and it is not true without further assumptions. The following example strongly suggests that the correct extra assumption needed for the equivalence to hold is that the forcing notion in question does not add reals.
Namely,
it is known to be consistent that $\CH$ holds and every Aronszajn tree is special (see \cite{MildenbergerShelah:SpecialisingAronszajnTrees}). But in a model of this theory, any ccc forcing preserves Aronszajn trees of height $\omega_1$ and any width (see the proof of \cite[Theorem 4.23]{FuchsMinden:SCforcingTreesGenAbs}), while the bounded forcing axiom for any forcing that adds a real fails (since it implies the failure of \CH; see \cite[Obs.~4.2(1)]{Fuchs:SCprinciplesAndDefWO}).

Thus, so far, we know that for a forcing notion $\P$, the $(\omega_1,\omega_1)$-Aronszajn tree preservation is equivalent to the bounded forcing axiom for $\P$, under two assumptions: the continuum hypothesis and that $\P$ does not add reals. We have argued above that the assumption that $\P$ does not add reals is indispensable here. I shall show that
in order for there to be a close connection between bounded forcing axioms and Aronszajn tree preservation, \CH seems less relevant than not adding reals. In fact, I will show that the assumption of \CH can be dropped in the abovementioned joint result with Minden, if it is formulated correctly.
The article is organized as follows.

Section \ref{sec:TheGeneralSetup} introduces three hierarchies of principles: $(\omega_1,\lambda)$-Aronszajn tree preservation, the bounded forcing axiom at $(\omega_1,\lambda)$, and $\Sigma^1_1(\omega_1,\lambda)$-absoluteness, for a forcing class $\Gamma$, and presents some known results and simple observations.

Section \ref{sec:Connections} explores the connections between the three hierarchies. Theorem \ref{thm:EquivalenceBtwBFAandSigma1-1-absoluteness} states that the bounded forcing axiom at $(\omega_1,\lambda)$ is equivalent to $\Sigma^1_1(\omega_1,\lambda)$ absoluteness for a forcing class $\Gamma$ and an uncountable cardinal $\lambda$.
Under the additional assumptions that forcing notions in $\Gamma$ don't add countable subsets to $\lambda$ and $\lambda^\omega=\lambda$, the main result, Theorem \ref{thm:CharacterizationOfATP(omega1,lambda)}, states that under these two conditions, these principles are also equivalent to $(\omega_1,\lambda)$-Aronszajn tree preservation.

In the remaining two sections of the article, I work with a concrete forcing class, the class of subcomplete forcing notions. This class was introduced by Jensen \cite{Jensen:SPSCF},  \cite{Jensen2014:SubcompleteAndLForcingSingapore}, and fits perfectly into the context of Aronszajn tree preservation, because subcomplete forcing notions do not add reals, but are iterable, and contain interesting forcing notions such as Namba forcing, \Prikry forcing, all countably closed forcing notions, etc.

In Section \ref{sec:SubclassesOfSCforcing}, the focus is on subclasses of subcomplete forcing notions, namely those that are subcomplete above $\lambda$. I show that these forcing notions are $[\lambda]^\omega$-preserving, and the three properties mentioned above are equivalent for this class, if $\lambda^\omega=\lambda$. It follows that the three properties for subcomplete forcing at $2^\omega$ are equivalent (this is the generalization I was aiming for). I show that while there are iteration theorems for forcing notions that are subcomplete above $\lambda$, their bounded forcing axioms behave differently from those for other forcing classes: Observation \ref{obs:Impossibility} shows that one cannot, in general, force the bounded forcing axiom for subcomplete forcing above $\omega_2$ at $(\omega_1,\omega_2)$ by a forcing notion that is subcomplete above $\kappa$, a strongly uplifting cardinal, and collapses $\kappa$ to $\omega_2$. The corresponding fact holds for the class of proper, or subcomplete, forcing notions, among others.

Finally, in Section \ref{sec:ATPbySCforcing}, I analyze $(\omega_1,\lambda)$-Aronszajn tree preservation under subcomplete forcing systematically, depending on where $\lambda$ lies in relation to $2^\omega$. For $\lambda<2^\omega$, the property is provable in $\ZFC$, as was shown in \cite{FuchsMinden:SCforcingTreesGenAbs}. For $\lambda=2^\omega$, it is equivalent to the bounded forcing axiom for subcomplete forcing at $(\omega_1,\lambda)$. 
Further, I determine the consistency strength of $(\omega_1,(2^\omega)^+)$-Aronszajn tree preservation to be an uplifting cardinal, and I obtain $\AD^{L(\reals)}$ as a lower bound on the consistency strength of $(\omega_1,(2^\omega)^{++})$-Aronszajn tree preservation. All of these consistency strength calculations also apply to the restricted class of subcomplete forcing notions that are countably distributive. I end with some open questions.

\section{Three hierarchies}
\label{sec:TheGeneralSetup}

In this section, I introduce the three hierarchies of interest and explore the relationships between them in a general setting, that is, without referring to any concrete classes of forcing notions.

\subsection{Aronszajn tree preservation}
\label{subsec:ATP}

The main objects of study here are trees of given height and width. Of main interest are trees of height $\omega_1$. Classically, much work has been done on $\omega_1$-trees, that is, trees of height $\omega_1$ all of whose levels are countable. More flexibility is introduced as follows.

\begin{defn}
\label{def:TreesWithWidths}
Let $\kappa$ and $\lambda$ be ordinals. A tree $T$ is a \emph{$(\kappa, {\leq}\lambda)$-tree} if $T$ is a tree of height $\kappa$ with levels of size less than or equal to $\lambda$. I refer to the restriction on the size of the levels of the tree in the second coordinate as the tree's \emph{width}, so that a $(\kappa, {\leq}\lambda)$-tree has width ${\le}\lambda$. I will sometimes drop the ``$\le$'', so that a $(\kappa,\lambda)$-tree is a $(\kappa,{\le}\lambda)$-tree.

Similarly, a \emph{$(\kappa,{<}\lambda)$-tree} is a tree with height $\kappa$ all of whose nonempty levels have size less than $\lambda$, and the width of such a tree is ${<}\lambda$.

A \emph{cofinal branch} in such a tree is a downward closed set of nodes that, when equipped with the restriction of the tree order, forms a well-order of type $\kappa$.

A \emph{$(\kappa,{\le}\lambda)$-Aronszajn tree} is a $(\kappa,{\le}\lambda)$-tree with no cofinal branch, and similarly, a \emph{$(\kappa,{<}\lambda)$}-Aronszajn tree is a \emph{$(\kappa,{<}\lambda)$}-tree with no cofinal branch.
\end{defn}

The key property of interest in this article is the following.

\begin{defn}
\label{defn:ATP}
Let $\Gamma$ be a forcing class, and let $\kappa,\lambda$ be ordinals. Then the \emph{strong Aronszajn preservation principle}, denoted  \emph{\sATP{\kappa,\lambda}{\Gamma},} says that forcing with any forcing notion in $\Gamma$ preserves $(\kappa,{\le}\lambda)$-Aronszajn trees. The principle \sATP{\kappa,{<}\lambda}{\Gamma} is defined similarly.

The weaker form of the principle, the \emph{Aronszajn tree preservation principle}, \wATP{\kappa,\lambda}{\Gamma}, says that whenever $T$ is a $(\kappa,{\le}\lambda)$-Aronszajn tree and $\P$ is a forcing notion in $\Gamma$, there is a $p\in\P$ such that $p\forces_\P$ ``$T$ is a $(\kappa,{\le}\lambda)$-Aronszajn tree.'' Again, \wATP{\kappa,{<}\lambda}{\Gamma} is defined similarly.
\end{defn}

As stated at the outset, the case of main interest here is $\kappa=\omega_1$. Note that the principle $\sATP{\omega,\lambda}{\Gamma}$ holds for any $\lambda$ and any forcing class $\Gamma$, because a tree of height $\omega$ is Aronszajn iff its reversed order is well-founded, and well-foundedness is absolute.

The difference between \sATP{\kappa,{\le}\lambda}{\Gamma} and its weak variant is subtle, and in most naturally encountered cases, these principles are equivalent. I'll introduce some language to make this more precise.

\begin{defn}
\label{defn:ForcingEquivalenceAndNaturalness}
Given a forcing notion $\P$ and a condition $p\in\P$, I write $\P_{{\le}p}$ for the restriction of $\P$ to conditions $q\le p$.

Two forcing notions $\P$ and $\Q$ are \emph{forcing equivalent} if they give rise to the same forcing extensions.

A class $\Gamma$ of forcing notions is \emph{natural} if for every $\P\in\Gamma$ and every $p\in\P$, there is a $\Q\in\Gamma$ such that $\P_{{\le}p}$ is forcing equivalent to $\Q$.
\end{defn}

It is easy to see that for a natural class $\Gamma$ of forcing notions, \sATP{\kappa,\lambda}{\Gamma} and \wATP{\kappa,\lambda}{\Gamma} are equivalent. The following easy observation motivates much of the present work:

\begin{obs}[Folklore]
\label{obs:CCforcingPreservesAronszajnTreesOfAnyWidth}
Countably closed forcing preserves any $(\omega_1,\lambda)$-Aronszajn tree, for any $\lambda$. In other words, the strong Aronszajn tree preservation principle at $(\omega_1,\lambda)$ for countably closed forcing,
$\sATP{\omega_1,\lambda}{\CC}$, holds for every cardinal $\lambda$.
\end{obs}

A stronger form of Aronszajn tree preservation would be the property of not adding a new branch to any $(\kappa,\lambda)$-tree $T$. Let's call this property $[T]$-preservation. It is easy to see that the previous observation does not admit this strengthening in general.

\begin{obs}
\label{obs:Add(omega1,1)AddsBranch}
Countably closed forcing may add a cofinal branch to an $(\omega_1,{\le}2^{\omega})$-tree.
\end{obs}

Namely, the forcing to add a Cohen subset to $\omega_1$ adds a new branch to the binary tree ${}^{{<}\omega_1}2$. However, countably closed forcing cannot add a cofinal branch to an $(\omega_1,{<}2^\omega)$-tree, and in fact, this generalizes to the class of subcomplete forcing, see Theorem \ref{thm:SCforcingAddsNoNewBranchGeneral}.

\subsection{Bounded forcing axioms}
\label{subsec:BFA}

The second concept of interest for this work is the bounded forcing axiom, originally introduced in Goldstern-Shelah \cite{GoldsternShelah:BPFA} for proper forcing:

\begin{defn}
\label{def:BoundedForcingAxioms}
Let $\Gamma$ be a class of forcing notions, and let $\kappa,\lambda$ be cardinals. Then $\BFA_\Gamma(\kappa,\lambda)$ is the statement that if $\P$ is a forcing notion in $\Gamma$, $\B$ is its complete Boolean algebra, and $\mathcal{A}$ is a collection of at most $\kappa$ many maximal antichains in $\B$, each of which has size at most $\lambda$, then there is an $\mathcal{A}$-generic filter in $\B$, that is, a filter that intersects each antichain in $\mathcal{A}$.
The versions $\BFA_\Gamma(\kappa,{<}\lambda)$ of these principles have the obvious meanings. The most well-known case is where $\kappa=\lambda=\omega_1$, and so, $\BFA_\Gamma$ stands for $\BFA_\Gamma(\omega_1,\omega_1)$.
\end{defn}

The following useful characterization of these axioms is easily seen to be equivalent to the one given in \cite[Thm.~1.3]{ClaverieSchindler:AxiomStar}, see also \cite{BagariaGitmanSchindler:RemarkableWeakPFA}. Here and in the following, if $M$ is a model of a first or second order language, then I write $|M|$ for the universe of $M$.

\begin{fact}
\label{fact:CharacterizationOfBFAkappa}
For a forcing notion $\P$, $\BFA_{\{\P\}}(\omega_1,\lambda)$ is equivalent to the following statement:
if $M=\langle|M|,\in,R_0,R_1,\ldots,R_i,\ldots\rangle_{i<\omega_1}$ is a transitive model for the language of set theory with $\omega_1$ many predicate symbols $\seq{\dot{R}_i}{i<\omega_1}$, of size $\lambda$, and $\phi(x)$ is a $\Sigma_1$-formula such that $\forces_\P\phi(\check{M})$, then there are in $\V$ a transitive $\bM=\kla{|\bM|,\in,\bar{R}_0,\bar{R}_1,\ldots,\bar{R}_i,\ldots}_{i<\omega_1}$ and an elementary embedding $j:\bM\prec M$ such that $\phi(\bM)$ holds.
\end{fact}

It will turn out that under certain conditions, the characterization of $\BFA_{\{\P\}}(\omega_1,\lambda)$ provided by this fact corresponds to \wATP{\omega_1,\lambda}{\{\P\}}. If one changes the requirement that $\forces_\P\phi(\check{M})$ to just say that there is a $p\in\P$ such that
$p\forces_\P\phi(\check{M})$, then one obtains a strong version of $\BFA(\{\P\},{\le}\lambda)$ that would then correspond to the \sATP{\omega_1,\lambda}{\{\P\}}.

\subsection{$\Sigma^1_1$-absoluteness}

The third property of interest, which will mainly serve as an intermediary between the other two, is the following two-cardinal version of $\Sigma^1_1$-absoluteness. Again, in this article, the first of the two cardinals will usually be $\omega_1$.

\begin{defn}
\label{defn:TwoCardinalSigma1-1Absoluteness}
Let $\Gamma$ be a forcing class, and let $\kappa\le\lambda$ be cardinals. Then \emph{$\Sigma^1_1(\kappa,\lambda)$-absoluteness for $\Gamma$} is the following statement: if $M=\kla{|M|,R_0,R_1,\ldots,R_\xi,\ldots}_{\xi<\kappa}$ is a model of a  first order language $\mathcal{L}$ with $\kappa$ many relation symbols, the cardinality of $|M|$ is $\lambda$, $\phi$ is a $\Sigma^1_1$-sentence over $\mathcal{L}$, and $\P\in\Gamma$ forces that $M\models\phi$, then in $\V$, there is an $\bM\prec M$ (so $\bM$ is an elementary submodel of $M$ with respect to first order formulas) such that $\bM\models\phi$ (this is second order satisfaction).
\end{defn}

Again, one could define a strong version of this principle in which one only assumes that some $p\in\P$ forces that $M\models\phi$. This version would then correspond to strong Aronszajn tree preservation/the strong bounded forcing axiom.

\begin{obs}
\label{obs:TwoCardinalSigma1-1absForCCforcing}
For any cardinal $\lambda$, $\Sigma^1_1(\omega_1,\lambda)$-absoluteness for countably closed forcing holds.
\end{obs}

\begin{proof}
Let $\P$ be a countably closed forcing notion, and let $M=\kla{\lambda,\vec{R}}$, where $\vec{R}=\kla{R_\xi\st\xi<\omega_1}$ is a sequence of relations on $\lambda$. We may assume that $R_0={\in}\rest\lambda$. Let's assume that forcing with $\P$ makes some $\Sigma^1_1$ statement $\psi$ in the second order language over $M$ true. There is then some first order formula $\phi$ in the language of $M$ with an extra unary predicate symbol $\dot{A}$ such that the fact that after forcing with $\P$, $M\models\psi$ can be expressed by the assertion that after forcing with $\P$, there is an $A\sub\lambda$ such that $\kla{M,A}\models\phi$. Let $\tau$ be a $\P$-name for such an $A$. Note that in $\V^\P$, the structure $\kla{\check{M},\tau}$ has a canonical set of Skolem functions. It is thus straightforward, using the countable closure of $\P$, to construct in $\V$ a sequence $\kla{\kla{p_i,X_i,A_i}\st i<\omega_1}$ where $\vec{p}$ is a weakly decreasing sequence of conditions in $\P$, $\vec{X}$ and $\vec{A}$ are sequences of subsets of $\lambda$ weakly increasing with respect to inclusion, and for all $i<\omega_1$, $p_i$ forces that $\check{X}_i$ is the least subset of $\lambda$ such that $\kla{\lambda,\vec{R}\rest i,\tau}|X\prec\kla{\lambda,\vec{R}\rest i,\tau}$ and that $\tau\cap\check{X}_i=\check{A}_i$. It is then easy to check that setting $A=\bigcup_{i<\omega_1}A_i$, $X=\bigcup_{i<\omega_1}X_i$ and $\bM=M|X$ (the restriction of $M$ to $X$), it follows that $\bM\prec M$ and $\kla{\bM,A}\models\phi$, which means that $\bM\models\psi$, as required by Definition \ref{defn:TwoCardinalSigma1-1Absoluteness}.
\end{proof}

The following is a transitivity property for $\Sigma^1_1$ absoluteness. The analog holds for bounded forcing axioms as well, but that will not be needed here.

\begin{lem}
\label{lem:ComposingWithSigma1-1-absoluteForcingIsFree}
Let $\lambda$ be a cardinal, $\P$ a forcing notion and $\dot{\Q}$ a $\P$-name such that $\forces_\P$``$\Sigma^1_1$-absoluteness for $\dot{\Q}$ holds.''
Then the following are equivalent:
\begin{enumerate}[label=(\arabic*)]
\item
\label{item:Sigma1-1absForP}
$\Sigma^1_1(\omega_1,\lambda)$-absoluteness for $\{\P\}$,
\item
\label{item:Sigma1-1absForP*Qdot}
$\Sigma^1_1(\omega_1,\lambda)$-absoluteness for $\{\P*\dot{\Q}\}$.
\end{enumerate}
\end{lem}

\begin{proof}
The substantial direction is \ref{item:Sigma1-1absForP}$\implies$\ref{item:Sigma1-1absForP*Qdot}. Suppose that $M$ is a structure with universe $L_\lambda$ (any set of size $\lambda$ works), for a first order language with $\omega_1$ many symbols and $\psi$ is a $\Sigma^1_1$ sentence over that language such that $\P*\dot{\Q}$ forces that $M\models\psi$. We may assume that the symbols in the language are the predicate symbols $\seq{\dot{R}_i}{i<\omega_1}$, that $\dot{R}_i$ is interpreted as $R_i$ in $M$, and by adding more predicate symbols if necessary, we may assume that for every $a\in L_{\omega_1}$, there is a $\xi<\omega_1$ such that $R_\xi=\{a\}$.

Let $G*H$ be arbitrary $\P*\dot{\Q}$-generics. Then $H$ is generic over $\V[G]$ for $\dot{\Q}^G$, and since by assumption, in $\V[G]$, $\Sigma^1_1(\omega_1,\lambda)$-absoluteness for $\dot{\Q}^G$ holds,
it follows that in $\V[G]$ there is an $X\sub M$ such that $M|X\prec M$ and $M|X\models\psi$.
Note that by our assumption on the predicates in the language of $M$, it follows that $L_{\omega_1}\sub X$.

We can now express the existence of such an $X$ as a $\Sigma^1_1$ statement over the structure $M^+$, which is $M$ equipped with a truth predicate for all formulas of the language of $M$. Writing $T$ for this truth predicate, we have that for every G\"{o}del number $\code{\phi}$ of a formula of the language of $M$, and every tuple $\vec{a}$ of the right arity, $M^+\models T(\code{\phi},\kla{\vec{a}})$ iff $M\models\phi(\vec{a})$. We may organize it so that $\code{\phi}\in L_{\omega_1}$. The existence of an $X$ as above can now be expressed in a $\Sigma^1_1$ way over $M^+$ by saying: there are a $Y$ and an $A$ (second order quantifications over $M$) such that $Y\neq\leer$ and for every tuple $\vec{a}\in Y$ and every formula $\code{\exists v\phi(\vec{x})}$ of matching arity, if $T(\code{\exists v\phi,\kla{\vec{a}}})$ holds, then there is a $b\in Y$ with $T(\phi,\kla{b,\vec{a}})$. This expresses that $M|Y\prec M$, by the Tarsk\'{\i}-Vaught criterion. To express that $M|Y\models\psi$ (this is second order satisfaction), one can just say that $A\sub Y$ and $\psi^Y(A)$ holds, the relativization of $\psi$ to $Y$. Let's denote this $\Sigma^1_1$ statement over $M^+$ by $\chi$. Note that $M^+\in\V$, and in $\V[G]$, $M^+\models\chi$. Thus, by \ref{item:Sigma1-1absForP}, there is a $Z$ in $\V$ such that $M^+|Z\prec M^+$ and $M^+|Z\models\chi$. Let $A,Y$ witness that $M^+|Z\models\chi$. Then $A\sub Y\sub Z$. By our assumption on $\vec{R}$, it follows that $L_{\omega_1}\sub Z$, so that every G\"{o}del number of a formula of the language of $M$ is in $Z$. As a result, since $M^+\models\chi$, and letting $\bM=M|Y$, it follows that $\bM\prec M$, and $\chi$ explicitly states that $\bM\models\psi$. So $\bM$ is as required by Definition \ref{defn:TwoCardinalSigma1-1Absoluteness}.

The converse direction \ref{item:Sigma1-1absForP*Qdot}$\implies$\ref{item:Sigma1-1absForP} is  immediate: if $M$ is a model of size $\lambda$ of a first order language of size $\omega_1$ such that in $\V^\P$, some $\Sigma^1_1$ formula $\psi$ is true in $M$, then by upwards absoluteness, this is still true in $\V^{\P*\dot{\Q}}$, and so, by \ref{item:Sigma1-1absForP*Qdot}, there is an $\bM\prec M$ in $\V$ with $\bM\models\psi$, as wished.
\end{proof}

Note that the implication \ref{item:Sigma1-1absForP*Qdot}$\implies$\ref{item:Sigma1-1absForP} in the previous lemma holds in general, for arbitrary $\P$ and $\dot{\Q}$, where $\dot{\Q}$ is a $\P$-name for a notion of forcing. The following corollary is an immediate consequence of Observation \ref{obs:CCforcingPreservesAronszajnTreesOfAnyWidth} and Lemma \ref{lem:ComposingWithSigma1-1-absoluteForcingIsFree}.

\begin{cor}
\label{cor:ComposingWithCCforcingIsFree}
Let $\P$ be a forcing notion, $\dot{\Q}$ a $\P$-name for a countably closed forcing notion and $\lambda$ a cardinal. Then the following are equivalent:
\begin{enumerate}[label=(\arabic*)]
\item
\label{item:Sigma1-1absForPagain}
$\Sigma^1_1(\omega_1,\lambda)$-absoluteness for $\{\P\}$,
\item
\label{item:Sigma1-1absForP*QdotAgain}
$\Sigma^1_1(\omega_1,\lambda)$-absoluteness for $\{\P*\dot{\Q}\}$.
\end{enumerate}
\end{cor}

Note that the assumption that $\dot{\Q}$ is countably closed in $\V^\P$ is not needed for the implication \ref{item:Sigma1-1absForP*QdotAgain}$\implies$\ref{item:Sigma1-1absForPagain} in the previous corollary.

\section{Connections between the hierarchies}
\label{sec:Connections}

In order to establish a close connection between $\Sigma^1_1$-absoluteness and the property characterizing bounded forcing axioms stated in Fact \ref{fact:CharacterizationOfBFAkappa}, I will use a translation procedure between first order truth in $H_{\lambda^+}$ and second order truth in $L_\lambda$. 
This kind of translation is part of the folklore, but the details are important here, so I will include the construction. For a less general prototype, see \cite{FuchsMinden:SCforcingTreesGenAbs}. The following definition allows us to code members of $H_{\lambda^+}$ by subsets of $\lambda\times\lambda$.

\begin{defn}
\label{defn:lambda-code}
Let $\lambda$ be an ordinal.
A $\lambda$-\emph{code} is a pair $\kla{R,\alpha}$, where $R\subset\lambda\times\lambda$, $\alpha<\lambda$ and $\kla{\lambda,R}$ is extensional and well-founded.

If $\kla{R,\alpha}$ is a $\lambda$-code, then let $U_R$, $\sigma_R$ be the unique objects (given by Mostowski's isomorphism theorem) such that $U_R$ is transitive and $\sigma_R:\kla{U_R,{\in}\rest U_R}\To\kla{\lambda,R}$ is an isomorphism. The \emph{set coded by $\kla{R,\alpha}$} is
\[c_{R,\alpha}=\sigma_R^{-1}(\alpha).\]
\end{defn}

Clearly, every member of $H_{\lambda^+}$ is coded by a $\lambda$-code, and every set that's coded by a $\lambda$-code is a member of $H_{\lambda^+}$. Using codes, $\Sigma_0$ statements over $\kla{H_{\lambda^+},\in}$ can essentially be translated into $\Sigma^1_1$ statements over $\lambda$, if one equips $\lambda$ with a predicate $E$ so that $\kla{\lambda,E}$ satisfies a sufficient rudimentary fragment of set theory.%
\footnote{Here and at many places to follow, for ease in readability, when $U$ is a transitive set or class, I write $\kla{U,{\in}}$, when I really mean $\kla{U,{\in}\rest U}$.}
I find it convenient to work with $L_{\lambda}$ instead of $\lambda$ here. In the statement of the following observation, when writing $\lambda^+$, I mean the least cardinal greater than $\lambda$, even if $\lambda$ itself may not be a cardinal.

\begin{obs}
\label{obs:Translation}
Let $\varphi(v_0,\ldots,v_{n-1})$ be a $\Sigma_0$-formula in the language of set theory. Then there is a $\Sigma^1_1$-formula $\varphi^c$ in the corresponding second order language, with free variables $X_0,x_0\ldots,X_{n-1},x_{n-1}$ (upper case variables being second order and lower case ones being first order) such that the following holds:

Whenever $\lambda$ is an ordinal such that $L_\lambda\models\ZFCm$,
$a_0,\ldots,a_{n-1}\in H_{\lambda^+}$ and $\kla{R_0,\alpha_0},\ldots,\kla{R_{n-1},\alpha_{n-1}}$ are $\lambda$-codes such that $a_i$ is coded by $\kla{R_i,\alpha_i}$, for $i<n$, then
\[\kla{H_{\lambda^+},\in}\models\varphi(a_0,\ldots,a_{n-1}) \iff
\kla{L_{\lambda},\in}\models\varphi^c(R_0,\alpha_0,\ldots,R_{n-1},\alpha_{n-1}).\]
Note that the satisfaction relation on the left is first order while the one on the right is second order.
\end{obs}

\begin{proof}
The construction of $\varphi^c$ proceeds by recursion on $\varphi$. I will assume that $\varphi$ is presented in such a way that the only subformulas of $\varphi$ that are negated are atomic. Any formula can be written in this form. In the following, lower/upper case variables will always stand for first/second order variables.

If $\varphi$ is of the form $v_0=v_1$, then $\varphi^c(X_0,x_0,X_1,x_1)$ is defined in such a way that it expresses: 
there is an injective function $F:\lambda\To\lambda$ with $F(x_0)=x_1$, such that whenever $\beta_0 X_0 \beta_1\ldots X_0 \beta_m X_0 x_0$, then $F(\beta_0)X_1F(\beta_1)\ldots X_1F(\beta_m)R_1x_1$ and vice versa. Expressing the existence of such a function requires a second order existential quantification. Hence, the resulting formula $\varphi^c(X_0,x_0,X_1,x_1)$ can be written as a $\Sigma^1_1$ formula.

If $\varphi$ is of the form $v_0\in v_1$, then $\varphi^c(X_0,x_0,X_1,x_1)$ is defined to express: 
there is a $\beta<\lambda$ such that $\beta X_1 x_1$, and such that the sentence of the form $(v_0=v_1)^c$ holds of $X_0,x_0,X_1,\beta$ (reducing to the previous case). The second order existential quantification occurring in $(v_0=v_1)^c$ can be pushed in front of the first order quantification (``there exists a $\beta<\lambda$''), in this case simply because both are existential quantifications.

If $\varphi$ is of the form $\neg(v_0=v_1)$, then $\varphi^c(X_0,x_0,X_1,x_1)$ is defined to express: 
there are $U_0,U_1,F$ such that $U_0$ is closed under $X_0$-predecessors, $U_1$ is closed under $X_1$-predecessors and $F:\kla{U_0,X_0\cap U_0^2}\To\kla{U_1,X_1\cap U_1^2}$ is a maximal isomorphism, meaning that $F$ cannot be expanded beyond $U_0$, and it is not the case that $x_0\in U_0$, $x_1\in U_1$ and $F(x_0)=x_1$. Here, $F:\kla{U_0,X_0\cap U_0^2}\To\kla{U_1,X_1\cap U_1^2}$ being a maximal isomorphism is expressible in a first order way as follows: for any $z\in\lambda\ohne U_0$ and any $z'\in\lambda\ohne U_1$, if one lets $U_0'=U_0\cup\{z\}$, $U_1'=U_1\cup\{z'\}$ and defines $F':U_0'\To U_1'$ by $F'\rest U_0=F$ and $F'(z)=z'$, then it is not the case that (a) $U_0'$ is closed under $X_0$-predecessors, (b) $U_1'$ is closed under $X_1$-predecessors and (c) $F':\kla{U'_0,X_0\cap {U'_0}^2}\To\kla{U'_1,X_1\cap {U'_1}^2}$ is an isomorphism.

If $\varphi$ is of the form $\neg(v_0\in v_1)$, then this can be expressed equivalently by $\forall v\in v_1 \neg (v_0=v)$. We already know how to translate $\neg(v_0=v)$, and we can then use the definition in the case of bounded quantification below.

The inductive steps corresponding to the logical connectives $\land$ and $\lor$ can be dealt with in the obvious way, setting $(\varphi\land\psi)^c=\varphi^c\land\psi^c$ and $(\varphi\lor\psi)^c=\varphi^c\lor\psi^c$.

Let's look at the case that $\varphi$ is of the form $\forall u\in w\quad\psi(u,w,v_0,\ldots,v_{n-1})$. Define the formula \[\varphi^c(Y,y,X_0,x_0,\ldots,X_{n-1},x_{n-1})\]
to express: for all $\beta Y y$, the formula $\psi^c$ is true of $Y,\beta,Y,y,X_0,x_0,\ldots,X_{n-1},x_{n-1}$. The resulting formula has a universal first order quantification over a $\Sigma^1_1$ formula. Since $\lambda$-sequences of subsets of $\lambda$ can be coded by single subsets of $\lambda$, the second order quantification can be pulled out in front of the first order quantifier, resulting in a $\Sigma^1_1$ formula. Encoding $\lambda$-sequences of subsets of $\lambda$ by subsets of $\lambda$ can be done uniformly, that is, the formula describing this process does not depend on $\lambda$, because $L_\lambda\models\ZFCm$, and so, we can use any preferred method, for example using \Goedel pairs.

The case of existential bounded quantification is easier, so I omit it here.
\end{proof}

\begin{thm}
\label{thm:EquivalenceBtwBFAandSigma1-1-absoluteness}
Let $\P$ be a notion of forcing, and let $\lambda$ be an uncountable cardinal. The following are equivalent:
\begin{enumerate}[label=(\arabic*)]
\item
\label{item:GeneralBFA}
$\BFA_{\{\P\}}(\omega_1,\lambda)$,
\item
\label{item:GeneralSigma1-1Absoluteness}
$\Sigma^1_1(\omega_1,\lambda)$-absoluteness for $\{\P\}$.
\end{enumerate}
\end{thm}

\begin{proof}
For the direction \ref{item:GeneralBFA}${\implies}$\ref{item:GeneralSigma1-1Absoluteness}, let
$M=\kla{|M|,R_0,R_1,\ldots,R_\xi,\ldots}_{\xi<\omega_1}$ be a model of a  first order language $\mathcal{L}$ with $\omega_1$ many relation symbols, such that the cardinality of $|M|$ is $\lambda$, and let $\phi$ be a $\Sigma^1_1$-sentence over $\mathcal{L}$ such that $\P$ forces that $M\models\phi$. We may assume that $|M|$ is an ordinal less than or equal to $\lambda$. Let $X\prec H_{\lambda^+}$ with $M\in X$, $X$ transitive, $X\in H_{\lambda^+}$. Consider the structure $N=\kla{X,\in,M,|M|,R_0,\ldots,R_\xi,\ldots}_{\xi<\omega_1}$. Then the statement that $M\models\phi$, being $\Sigma^1_1$, can be expressed as a $\Sigma_1$ statement $\phi'(N)$ which is forced to be true by $\P$. By $\BFA_{\{\P\}}(\omega_1,\lambda)$, there is in $\V$ a $j:\bN\prec N$, where $\bN$ is transitive and $\phi'(\bN)$ holds. Clearly, $\bN$ is of the form $\kla{\bX,\in,\bM,|\bM|,\bR_0,\ldots,\bR_\xi,\ldots}_{\xi<\omega_1}$. It is now obvious that $j``|\bM|$ is as required.

For the direction \ref{item:GeneralSigma1-1Absoluteness}${\implies}$\ref{item:GeneralBFA}, first observe that $\Sigma^1_1(\omega_1,\lambda)$-absoluteness for $\{\P\}$ implies that $\P$ does not collapse $\omega_1$ (that is, it is not the case that $\P$ forces that $\omega_1^\V$ is countable). This is because the existence of a surjection $f:\omega\To\omega_1^\V$ could easily be expressed as a $\Sigma^1_1(L_\lambda)$ statement. If we equip $L_\lambda$ with constant symbols for all countable ordinals, and call the resulting model $M$, then any $N\prec M$ that satisfies that $\Sigma^1_1$ statement in $\V$ would give rise to such a surjection in $\V$.

Let's now begin the proof. Let
$M=\kla{|M|,\in,R_1,\ldots,R_i,\ldots}_{1\le i<\omega_1}$ be a transitive model for the language of set theory with $\omega_1$ many predicate symbols $\seq{\dot{R}_i}{1\le i<\omega_1}$, of size $\lambda$, and let $\phi(x)$ be a $\Sigma_1$-formula such that $\forces_\P\phi(\check{M})$. We may assume that $M$ is closed under ordered pairs (otherwise, we may enlarge $M$ slightly to a model that is, and equip that model with a predicate for the universe of $M$), and so we may assume that the predicates (other than $\in$) are unary. We have to show that in $\V$, there are a transitive model $\bM=\kla{|\bM|,\in,\bar{R}_0,\bar{R}_1,\ldots,\bar{R}_i,\ldots}_{1\le i<\omega_1}$ of the same language and an elementary embedding $j:\bM\prec M$ such that $\phi(\bM)$ holds.

Since in $\V$, $M\in H_{\lambda^+}$, there is a $\lambda$-code $\kla{R,\alpha}$ for $M$.
Let $G$ be $\P$-generic over $\V$ such that $\omega_1^{\V[G]}=\omega_1^\V$. In $\V[G]$, since $M\in H_{\lambda^+}$ and the $\Sigma_1$-formula $\phi(M)$ holds, we actually have that $\kla{H_{\lambda^+},\in}^{\V[G]}\models\phi(M)$. Let $\phi$ be of the form $\exists w\quad \bar{\phi}(w,M)$, where $\bar{\phi}$ is a $\Sigma_0$ formula.

Let $H$ be $\Col{\omega_1,\lambda}^{\V[G]}$-generic over $\V[G]$. Since $\Sigma_1$-formulas are upward-absolute, it follows that
$\kla{H_{\lambda^+},\in}^{\V[G][H]}\models\exists w\quad\bar{\phi}(w,M)$ as well. And since $\Col{\omega_1,\lambda}^{\V[G]}$ is countably closed in $\V[G]$, $\omega_1^{\V}=\omega_1^{\V[G]}=\omega_1^{\V[G][H]}$. Note that $(\lambda^+)^{\V[G][H]}=\omega_2^{\V[G][H]}$.

Working in $\V[G][H]$, this means that there is an $\omega_1$-code for a witness $w$ to the fact that $\kla{H_{\omega_2},\in}^{\V[G][H]}\models\exists w\quad \bar{\phi}(w,M)$. Thus, there is an $\omega_1$-code $\kla{S,\delta}$ such that $\kla{L_\lambda,\in}\models\bar{\phi}^c(S,\delta,R,\alpha)$, $\bar{\phi}^c$ being the formula given by Observation \ref{obs:Translation}. I am being a little sloppy here, because officially, the translation of $\phi^c$ of $\phi$ from $H_{\lambda^+}$ to $L_\lambda$ requires us to work with $\lambda$-codes, but an $\omega_1$-code for a witness can be easily expanded in some trivial way to a $\lambda$-code, so let's not worry about this detail. The point is that the existence of such an $\omega_1$-code can be expressed in $L_\lambda$ (using $\omega_1$ as a parameter) in a $\Sigma^1_1$ way as follows: there are an $S\sub\omega_1\times\omega_1$ (this is a second order existential quantification) and a $\delta<\omega_1$ such that $\kla{\omega_1,S}$ is extensional (this is first order expressible), ``$\kla{\omega_1,S}$ is well-founded'' and $\bar{\phi}^c(S,\delta,R,\alpha)$ holds. Here, ``$\kla{\omega_1,S}$ is well-founded'' stands for the statement that there is a function $F:\omega_1\times\omega_1\To\omega_1$ (this is second order) such that for every $\gamma<\omega_1$, if we define $f_\gamma:\gamma\To\omega_1$ by $f_\gamma(\xi)=F(\gamma,\xi)$, then $f_\gamma:\kla{\gamma,S\cap(\gamma\times\gamma)}\To\kla{\omega_1,<}$ is order preserving. This suffices, since $\omega_1$ is a cardinal of uncountable cofinality, and any ill-foundedness would be witnessed by a countably infinite decreasing sequence, and hence already be visible in some $\kla{\gamma,S\cap(\gamma\times\gamma)}$. In a slight abuse of notation, let me write $\phi^c(R,\alpha,\omega_1)$ for this formula.

Recall that $\sigma_R:\kla{U_R,\in\rest U_R}\To\kla{\lambda,R}$ is the Mostowski isomorphism and $\sigma_R^{-1}(\alpha)=c_{R,\alpha}=M$.
Note every $\xi<\lambda$ has a very simple $\lambda$-code, namely the code $\kla{<\rest\lambda,\xi}$. I will just write $\kla{<,\xi}$ for this code. Let's also fix a $\lambda$-code $\kla{C,0}$ for $\omega_1$ (in the case $\lambda=\omega_1$, $\kla{<,\omega_1}$ is not a $\lambda$-code, but of course there is a $\lambda$-code for $\omega_1$ in $\V$).

I want to view $M$ as a function, $M:\omega_1\To\V$, such that $M(0)=|M|$ and for $1\le\xi<\omega_1$, $M(\xi)=R_\xi$. For $\xi<\omega_1$, let $\zeta_\xi=\sigma_R(M(\xi))$.

Let $\chi(x,y,z)$ be a $\Sigma_0$-formula expressing that $x$ is a function and $x(y)=z$. Thus, for all $\xi<\omega_1$, $\chi(M,\xi,M(\xi))$ holds. This means that $c_{R,\alpha}(c_{<,\xi})=c_{R,\zeta_\xi}$, which implies that $\kla{L_\lambda,\in}\models\chi^c(R,\alpha,<,\xi,R,\zeta_\xi)$ (using the translation procedure of Observation \ref{obs:Translation}.) This is true in $\V$. So we can choose witnesses $W_\xi$ for these $\Sigma^1_1$ facts: if $\chi^c=\exists Z\quad\tilde{\chi}^c$, then we have for all $\xi<\omega_1$: $\kla{L_\lambda,\in}\models\tilde{\chi}^c(W_\xi,R,\alpha,<,\xi,R,\zeta_\xi)$.
I also would like to add a witness $W$ to the translation $(\chi')^c(R,\alpha,C,0)$ of the formula $\chi'(M)$ which expresses that $M$ is a function with domain $\omega_1$, $M(0)$ is transitive and for all $\xi<\omega_1$, $M(\xi)\sub M(0)$.

Consider the structure \[N=\kla{L_\lambda,\in,\omega_1,R,\alpha,W,\zeta_0,W_0,\{0\},\zeta_1,W_1,\{1\},\ldots,\zeta_\xi,W_\xi,\{\xi\},\ldots}_{\xi<\omega_1}.\]
We have that in $\V[G][H]$, $N\models\phi^c$. Note that I view $\phi^c$ as a sentence here, since the parameters $R,\alpha,\omega_1$ are available as predicates in the structure $N$.

By Corollary \ref{cor:ComposingWithCCforcingIsFree}, $\Sigma^1_1(\omega_1,\lambda)$-absoluteness holds for $\{\P*\dot{\mathsf{Col}}(\omega_1,\lambda)\}$. So let $\bN\in\V$ be such that $\bN\prec N$ and $\bN\models\phi^c$. Let $|\bN|$ be the universe of $\bN$. Since we added the predicates $\{\xi\}$, for $\xi<\omega_1$, it follows that $\omega_1\sub|\bN|$. So let
\[j:|\tN|\To|\bN|\]
be the inverse of the Mostowski collapse. Then $|\tN|=L_{\tlambda}\models\ZFCm$ and $j\rest\omega_1=\id$.

Let's expand $|\tN|$ to a structure such that
\[j:\tN\To\bN\]
is an isomorphism. So $\tN$ is of the form
\[\tN=\kla{L_\tlambda,\in,\omega_1,\tilde{R},\talpha,\tilde{W},\tilde{\zeta}_0,\tilde{W}_0,\tilde{\zeta}_1,\tilde{W}_1,\ldots,\tilde{\zeta}_\xi,\tilde{W}_\xi,\ldots}_{\xi<\omega_1}.\]
By elementarity, we have that
\[\kla{L_{\tlambda},\in}\models(\chi')^c(\tilde{R},\talpha,C,\omega_1)\]
and for every $\xi<\omega_1$
\[\kla{L_\tlambda,\in}\models\chi^c(\tilde{R},\talpha,<,\xi,\tilde{R},\tzeta_\xi)\]
because we explicitly added the witnessing subsets of $L_\lambda$ to the structure $N$.

Let $\bM=c_{\tilde{R},\talpha}$. Then since $\tN\models\phi^c$, it follows that $\phi(\bM)$ holds. This is because $\phi^c$ expressed that there are $S\sub\omega_1\times\omega_1$ and a $\delta<\omega_1$ such that $\kla{\omega_1,S}$ is extensional, ``$\kla{\omega_1,S}$ is well-founded'' and $\bar{\phi}(S,\delta,R,\alpha)$ holds, by saying that there is a function $F:\omega_1\times\omega_1\To\omega_1$ such that for every $\gamma<\omega_1$, defining $f_\gamma:\gamma\To\omega_1$ by $f_\gamma(\xi)=F(\gamma,\xi)$, we have that $f_\gamma:\kla{\gamma,S\cap\gamma\times\gamma}\To\kla{\omega_1,<}$ is order preserving. Since the correct $\omega_1$ is used in this formula, any witnessing $\kla{S,\delta}$ must be an $\omega_1$-code.

Moreover, since $\tN\models(\chi')^c(\tilde{R},\talpha,C,\omega_1)$, we know that $\bM$ is a function with domain $\omega_1$, $\bM(0)$ is transitive, and for all $\xi<\omega_1$, $\bM(\xi)\sub\bM(0)$. View $\bM$ as a model,
$\bM=\kla{|\bM|,\in,\vec{\bar{R}}}$.

Finally,
\[\sigma_R^{-1}\compose j\compose\sigma_{\tilde{R}}\rest|\bM|:\bM\To M\]
is elementary: to see this, let $\phi_0(\vx)$ be a formula in the language of $\bM$, and let $\va=a_0,\ldots,a_{n-1}\in|\bM|$. Let $\xi_i=\sigma_{\tilde{R}}(a_i)$. Since $\sigma_{\tilde{R}}(\bar{R}_\xi)=\tzeta_\xi$ and $\sigma_{\tilde{R}}(|\bM|)=\talpha$, we can interpret $\phi_0$ in $\tN$ by bounding every quantifier by $\{\gamma<\lambda\st\gamma\tilde{R}\talpha\}$ and interpreting the predicate $\dot{R}_\xi$ by the class $\{\gamma<\lambda\st\gamma\tilde{R}\tzeta_\xi\}$. Calling the resulting formula $\phi_1(\vx)$, we have that $\tN\models\phi_1(\vxi)$, i.e., $\tN\models\phi_1(\sigma_{\tilde{R}}(\va))$. By elementarity of $j$, it follows that $\bN$, and hence $N$, models that $\phi_1(j(\sigma_{\tilde{R}}(\va)))$ holds. But unraveling how $\phi_1$ was constructed from $\phi_0$, this means that $M\models\phi_0(\sigma_R^{-1}(j(\sigma_{\tilde{R}}(\va))))$.
\end{proof}

Towards establishing a connection between these equivalent principles and Aronszajn tree preservation, I will use the following terminology.

\begin{defn}
\label{defn:lambda-omega-preserving}
Let $\P$ be a notion of forcing, and let $\lambda$ be a cardinal. $\P$ is \emph{$[\lambda]^\omega$-preserving} if whenever $G$ is $\P$-generic over $\V$, then we have:
\[[\lambda]^\omega=([\lambda]^\omega)^{\V[G]}.\]
\end{defn}

The following theorem summarizes the connections between Aronszajn tree preservation, bounded forcing axioms and two-cardinal-$\Sigma^1_1$ absoluteness.

\begin{thm}
\label{thm:CharacterizationOfATP(omega1,lambda)}
Let $\lambda$ be a cardinal such that $\lambda^\omega=\lambda$, and let $\P$ be a $[\lambda]^\omega$-preserving forcing notion. Then the following are equivalent:
\begin{enumerate}[label=(\arabic*)]
\item
\label{item:BFA}
$\BFA_{\{\P\}}(\omega_1,\lambda)$,
\item
\label{item:Sigma1-1Absoluteness}
$\Sigma^1_1(\omega_1,\lambda)$-absoluteness for $\{\P\}$,
\item
\label{item:wATP}
$\wATP{\omega_1,\lambda}{\{\P\}}$.
\end{enumerate}
\end{thm}

\begin{proof}
By Theorem \ref{thm:EquivalenceBtwBFAandSigma1-1-absoluteness},  \ref{item:BFA} and \ref{item:Sigma1-1Absoluteness} are equivalent. Thus, it suffices to prove that \ref{item:BFA}$\implies$\ref{item:wATP} and that \ref{item:wATP}$\implies$\ref{item:Sigma1-1Absoluteness}.

\ref{item:BFA}$\implies$\ref{item:wATP}: Assuming \ref{item:BFA}, suppose that \ref{item:wATP} fails. This means that there is an $(\omega_1,{\le}\lambda)$-Aronszajn tree $T$ such that $\P$ forces that $T$ is not an $(\omega_1,{\le}\lambda)$-Aronszajn tree. We may assume that $T\sub\omega_1\times\lambda$. Let $M\prec H_{\lambda^+}$ be transitive, with $T\in M$, of size $\lambda$, and let $G$ be $\P$-generic. In $\V[G]$, $T$ has a branch of order type $\omega_1$. Let $M=\kla{M,\in,T,0,\ldots,\xi,\ldots}_{\xi<\omega_1}$. Then the existence of such a branch can be expressed as $\phi(M)$, where $\phi$ is a $\Sigma_1$-formula in the language of set theory. Since $\P$ forces $\phi(M)$, by $\BFA_{\{\P\}}(\omega_1,\lambda)$, there are in $\V$ a transitive $\bar{M}$ and an elementary $j:\bar{M}\prec M$ such that $\phi(\bar{M})$ holds. Letting $\bar{b}$ be a witness for this, it follows that $j``\bar{b}$ is a branch through $T$ of order type $\omega_1$ -- the point here is that $\omega_1\sub\bM$ and $j\rest\omega_1=\id$. Thus, $T$ is not Aronszajn in $\V$, a contradiction.

\ref{item:wATP}$\implies$\ref{item:Sigma1-1Absoluteness}:
Assume \wATP{\omega_1,\lambda}{\{\P\}}.
In order to verify that $\Sigma^1_1(\omega_1,\lambda)$-absoluteness for $\{\P\}$ holds,
let $M=\langle\lambda,R_0,R_1,\ldots,R_\xi,\ldots\rangle_{\xi<\omega_1}$ be an $\mathcal{L}$-structure, where $\mathcal{L}=\{\dot{R}_\xi\st\xi<\omega_1\}$, and each $\dot{R}_\xi$ is a relation symbol of finite arity, and $\dot{R}_\xi^M=R_\xi$.
Let $\phi$ be as in Definition \ref{defn:TwoCardinalSigma1-1Absoluteness}, and suppose that $\P$ forces that $M\models\phi$.

We have to find in $\V$ an $\bM\prec M$ such that $\bM\models\phi$. By renumbering, if necessary, we may assume that the only predicates occurring in $\phi$ are $R_0,\ldots, R_{n-1}$.
Let $\psi$ be a first order sentence in the language $\mathcal{L}$ with one additional predicate symbol $\dot{B}$ such that
\begin{enumerate}
\item[$(*)$]
$M\models\phi\iff \exists B\sub\lambda\quad \kla{\lambda,R_0,\ldots,R_{n-1},B}\models\psi$,
\end{enumerate}
where $B$ is the interpretation of $\dot{B}$ in this structure.

I will use the following notation: if $M$ is a model for a first order language $\mathcal{L}$ and $\bar{\mathcal{L}}\sub\mathcal{L}$, then $M\rest\bar{\mathcal{L}}$ is the reduct of $M$ to $\bar{\mathcal{L}}$. If $X\sub M$ and $M$ is a relational structure, then $M|X$ is the structure with universe $X$ in which the relations of $M$ are restricted to $X$.

For
$\alpha<\omega_1$, define languages
\[\mathcal{L}_\alpha=\{\dot{R}_i\st i<n\}\cup\{\dot{R}_\xi\st\xi<\alpha\}\ \text{and}\ \mathcal{L}_\alpha^+=\mathcal{L}_\alpha\cup\{\dot{B}\}.\]
In \V, consider the tree consisting of all functions $f$ with $\dom(f)\in\omega_1$ such that for all $\alpha\in\dom(f)$:
\begin{enumerate}[label=(\alph*)]
  \item
  \label{item:pairs}
  $f(\alpha)$ is of the form $\kla{x^f_\alpha,b^f_\alpha}$, where $b^f_\alpha\sub x^f_\alpha\in[\lambda]^\omega$,
  \item
  \label{item:elementarychain}
  letting
  \[M^f_\alpha=(M\rest\mathcal{L}_\alpha)|x^f_\alpha\ \text{and}\
    N^f_\alpha=(M^f_\alpha,b^f_\alpha),\]
      so that $M^f_\alpha$ is an $\mathcal{L}_\alpha$-structure and $N^f_\alpha$ is an $\mathcal{L}_\alpha^+$-structure ($\dot{B}$ is interpreted as $b^f_\alpha$ in the latter structure), we have: if $\beta\in\dom(f)$ and $\alpha\le\beta$, then
      \[N^f_\alpha\prec N^f_\beta\rest\mathcal{L}_\alpha^+\ \text{and}\ M^f_\alpha\prec M\rest\mathcal{L}_\alpha,\]
  \item
  \label{item:allthinkpsi}
  $N^f_\alpha\models\psi$.
\end{enumerate}
The tree ordering is inclusion. The size of $T$ is $(\lambda^\omega)^{{<}\omega_1}=\lambda$, since by assumption $\lambda^\omega=\lambda$.

Let $G\sub\P$ be generic. Then in $\V[G]$, $T$ has a cofinal branch, that is, there is in $\V[G]$ a function $f$ with domain $\omega_1$ such that for every $\alpha<\omega_1$, $f\rest\alpha\in T$. To see this, working in $\V[G]$, let $B$ be a witness to the fact that $(*)$ holds. Still in $\V[G]$, define an increasing sequence $\seq{x_\alpha}{\alpha<\omega_1}$ in $[\lambda]^\omega$ such that $(M,B)\rest\mathcal{L}_\alpha^+\prec (M,B)\rest\mathcal{L}_\alpha^+$.
Then $f(\alpha):=\kla{x_\alpha,B\cap x_\alpha}$ (for $\alpha<\omega_1$) is as wished.
By assumption, $\P$ is $[\lambda]^\omega$-preserving. 
Thus, it follows that for every $\alpha<\omega$, $f\rest\alpha\in\V$, since $f\rest\alpha$ is essentially a countable sequence of elements of $[\lambda]^\omega$. Hence, $f$ is a cofinal branch through $T$, $T$ has height $\omega_1$ in $\V$, and the size of $T$ (and hence also the width of $T$) is at most $\lambda$ in $\V$.

So $T$ is an $(\omega_1,{\le}\lambda)$-tree that's not Aronszajn in $\V[G]$. Since $G$ is an arbitrary $\P$-generic filter, this means that $\P$ forces that $T$ is not Aronszajn. By $\wATP{\omega_1,\lambda}{\{\P\}}$, it follows that $T$ is not Aronszajn in $\V$.
So let $g:\omega_1\To\V$ be a cofinal branch through $T$, $g\in\V$. For $\alpha<\omega_1$, let $g(\alpha)=\kla{x_\alpha,b_\alpha}$. Let $X=\bigcup_{\alpha<\omega_1}x_\alpha$ and $\bar{B}=\bigcup_{\alpha<\omega_1}b_\alpha$. Then $\bar{M}=M|X$ is as wished: Letting $\bar{N}=\kla{\bar{M},\bar{B}}$, we have that $\bar{N}\models\psi$ (by \ref{item:elementarychain} and \ref{item:allthinkpsi}), so that by $(*)$, $\bar{M}\models\phi$. Also by \ref{item:elementarychain}, $\bar{M}\prec M$, so we are done.
\end{proof}

Note that the implication \ref{item:BFA}$\implies$\ref{item:wATP} in the previous theorem goes through without assuming that $\P$ is $[\lambda]^\omega$-preserving. Let's make a note of this.

\begin{cor}
\label{cor:ATPcharacterizationTechnical}
Let $\lambda\ge\omega_1$ be a cardinal, and let $\P$ be a notion of forcing.
\begin{enumerate}[label=(\arabic*)]
\item
\label{item:BFAtoWATP}
$\BFA_{\{\P\}}(\omega_1,\lambda)$ implies $\wATP{\omega_1,\lambda}{\{\P\}}$,
\item
\label{item:BFAtoSigma1-1Absoluteness}
$\BFA_{\{\P\}}(\omega_1,\lambda)$ is equivalent to
$\Sigma^1_1(\omega_1,\lambda)$-absoluteness for $\{\P\}$.
\end{enumerate}
\end{cor}

The following is an immediate consequence of Theorem \ref{thm:CharacterizationOfATP(omega1,lambda)}.

\begin{cor}
\label{cor:CharacterizationOfATP(omega1,continuum)}
Let $\P$ be a forcing notion that does not add reals. Then the following are equivalent:
\begin{enumerate}[label=(\arabic*)]
\item
\label{item:BFA-NNR}
$\BFA_{\{\P\}}(\omega_1,2^\omega)$,
\item
\label{item:Sigma1-1Absoluteness-NNR}
$\Sigma^1_1(\omega_1,2^\omega)$-absoluteness for $\{\P\}$,
\item
\label{item:wATP-NNR}
$\wATP{\omega_1,2^\omega}{\{\P\}}$.
\end{enumerate}
\end{cor}

\begin{proof} This follows by applying Theorem
\ref{thm:CharacterizationOfATP(omega1,lambda)} to $\lambda=2^\omega$. Clearly then, $\lambda^\omega=\lambda$, and since $\P$ does not add reals, it follows that $\P$ is $[\lambda]^\omega$-preserving: using a bijection between $\power(\omega)$ and $\lambda$, an alleged new element of $[\lambda]^\omega$ can be viewed as a countable set of reals, which can be viewed as a single real, so it is not new.
\end{proof}

%

\section{Aronszajn tree preservation by subclasses of subcomplete forcing}
\label{sec:SubclassesOfSCforcing}

In this section, I will analyze the principles $\wATP{\omega_1,\kappa}{\Gamma}$ when $\Gamma$ is an appropriate subclass of subcomplete forcing. It will turn out that it will be crucial where $\kappa$ lies in comparison to the continuum. I will deal with each possible constellation in a separate subsection, but first, I will introduce the relevant classes of subcomplete forcing.

\subsection{Subcompleteness above $\mu$}

Jensen introduced the concept of subcompleteness above $\mu$ in \cite[\S2, pp.~47-49]{Jensen:SPSCF}, as follows (generalizing slightly). Following Jensen, a transitive model $\bN$ is \emph{full} if there is an ordinal $\delta>0$ such that $L_\delta(\bN)\models\ZFC^-$ and such that $\bN$ is \emph{regular} in $L_\delta(\bN)$, meaning that if $\gamma<\On\cap\bN$, $f:\gamma\To\bN$ and $f\in L_\delta(\bN)$, then $f\in\bN$.

\begin{defn}
\label{defn:SubcompleteAboveMu}
A notion of forcing $\P$ is \emph{subcomplete above an ordinal $\mu$} if for all sufficiently large $\theta$, we have that if
$\sigma:\bN\prec N$
where $N$ is of the form $L_\tau^A=\kla{L_\tau[A],\in,A}$, $H_\theta\sub N$, $N\models\ZFCm$, $\bN$ is countable, transitive and full, $\P,\mu\in\ran(\sigma)$, $a_0,\ldots,a_{n-1}\in\bN$
and $\bG$ is $\bP=\sigma^{-1}(\P)$-generic over $\bN$, then
there is a condition $p\in\P$ such that whenever $G\ni p$ is $\P$-generic, then in $\V[G]$, there is a $\sigma':\bN\prec N$ satisfying the following conditions:
\begin{enumerate}
\item $\sigma'(\bP)=\P$, $\sigma(a_i)=\sigma'(a_i)$, for all $i<n$, and $\sigma'\rest\bmu=\sigma\rest\bmu$ (where $\bmu=\sigma^{{-}1}(\mu)$),
\item $(\sigma')``\bG\sub G$,
\item $\Hull^N(\delta\cup\ran(\sigma))=\Hull^N(\delta\cup\ran(\sigma'))$, where $\delta=\delta(\P)$.
\end{enumerate}
If any mention of $\mu$ and $\bmu$ is removed in the previous definition, the result is the definition of subcompleteness.
\end{defn}

It is easy to see that every countably closed forcing is subcomplete above $\mu$ (for any $\mu$), because $\sigma'$ in the definition can be chosen to be equal to $\sigma$ in this case. Note that every subcomplete forcing is subcomplete above $\omega_1$. More generally:

\begin{obs}
\label{obs:SCaboveContinuumIsFree}
Every subcomplete forcing is subcomplete above $2^\omega$.
\end{obs}

\begin{proof}
In the situation of Definition \ref{defn:SubcompleteAboveMu}, let $f:\power(\omega)\To 2^\omega$ be the $<_A$-least such bijection. Then $f\in\ran(\sigma)$. Let $\barf=\sigma^{-1}(f)$. It then follows that $\sigma'(\barf)=\sigma(\barf)=f$, and so, for $\xi<(2^\omega)^\bN$, we have that
\[\sigma(\xi)=\sigma(\barf(\barf^{-1}(\xi)))=\sigma(\barf)(\sigma(\barf^{-1}(\xi)))=f(\sigma(\barf^{-1}(\xi))).\]
But since $\barf^{-1}(\xi)\sub\omega$, it follows that $\sigma(\barf^{-1}(\xi))=\sigma'(\barf^{-1}(\xi))$, and we can trace these identities backwards to arrive at $\sigma'(\xi)$.
\end{proof}

\begin{defn}
I write $\SCabove{\mu}$ for the class of forcing notions that are subcomplete above $\mu$, and $\SC$ stands for the class of subcomplete forcing.
\end{defn}

It was shown in \cite{Fuchs:ParametricSubcompleteness} that $\SC$ is natural, and the proof carries over to $\SCabove{\mu}$. Thus, the difference between Aronszajn tree preservation and its strong form disappears in the context of these forcing classes.
Moreover, in the following observation shows why this is class is particularly important in the present context.

\begin{obs}
\label{obs:SubcompleteAboveMuAddsNoOmegaSequenceInMu}
Suppose $\mu$ is a cardinal. Then forcing notions that are subcomplete above $\mu$ are $[\mu]^\omega$-preserving.
\end{obs}

\begin{proof}
Let $\P$ be subcomplete above $\mu$, and suppose towards a contradiction that there are a $p\in\P$ and a $\P$-name $\dot{a}$ such that $p$ forces with respect to $\P$ that $\dot{a}$ is a function from $\omega$ to $\check{\mu}$ that is not in $\check{\V}$. Let $\theta$ be sufficiently large, $H_\theta\sub L_\tau[A]$, $\sigma:\bN\prec N$ with $\sigma(\kla{\bp,\dot{\bar{a}},\bar{\P},\btheta,\bmu})=\kla{p,\dot{a},\P,\theta,\mu}$, where $\bN$ is countable and full (this can always be arranged). Let $\bG\ni\bp$ be $\bP$-generic over $\bN$, and let $q$ be a condition as guaranteed by the definition of subcompleteness above $\mu$. Let $G\ni q$ be $\P$-generic over $\V$, and let $\sigma':\bN\prec N$ be as in the definition. Then $\sigma'$ lifts uniquely to an elementary embedding from $\bN[\bG]$ to $N[G]$ that maps $\bG$ to $G$. Let's denote this embedding by $\sigma'$ as well. Let $\bar{a}=\dot{\bar{a}}^{\bG}$ and $a=\dot{a}^G$. Then $\sigma'(\bar{a})=a$ is a function from $\omega$ to $\mu$, and for $n<\omega$, $a(n)=\sigma'(\bar{a}(n))=\sigma(\bar{a}(n))$. So, since $\bar{a},\sigma\in\V$, so is $a$, a contradiction.
\end{proof}

This gives us the following version of Theorem \ref{thm:CharacterizationOfATP(omega1,lambda)}.

\begin{lem}
\label{lem:CharacterizationOfATP(omega1,mu)ForSCabovemu}
Let $\lambda$ be a cardinal such that $\lambda^\omega=\lambda$. Then the following are equivalent:
\begin{enumerate}[label=(\arabic*)]
\item
\label{item:BSCFAabovemu}
$\BFA_{\SCabove{\lambda}}(\omega_1,\lambda)$,
\item
\label{item:Sigma1-1AbsolutenessForSCabovemu}
$\Sigma^1_1(\omega_1,\lambda)$-absoluteness for subcomplete forcing above $\lambda$,
\item
\label{item:wATPforSCabovemu}
$\wATP{\omega_1,\lambda}{\SCabove{\lambda}}$.
\end{enumerate}
\end{lem}

\begin{proof}
This is an immediate consequence of Observation \ref{obs:SubcompleteAboveMuAddsNoOmegaSequenceInMu} and Theorem \ref{thm:CharacterizationOfATP(omega1,lambda)}.
\end{proof}

The following is an iteration theorem for subcomplete forcing above $\mu$, due to Jensen.

\begin{thm}[Jensen {\cite[\S 3, p.~5, Thm.~3]{Jensen:IterationTheorems}}]
\label{thm:IteratingSubcompletenessAboveMu}
Let $\seq{\B_i}{i<\alpha}$ be a revised countable support iteration of complete Boolean algebras and let $\seq{\mu_i}{i<\alpha}$ be a weakly increasing sequence such that for all $i+1<\alpha$, the following conditions are satisfied:
\begin{enumerate}
  \item $\B_i\neq\B_{i+1}$,
  \item $\forces_{\B_i}(\check{\B_{i+1}}/\dot{G}_{\B_i}\ \text{is subcomplete above $\check{\mu}_i$})$,
  \item $\forces_{\B_{i+1}}(\check{\B}_i\ \text{has cardinality $\le\check{\mu}_i$})$.
\end{enumerate}
Then for every $i<\alpha$, $\B_i$ is $\mu_0$-subcomplete.
\end{thm}

In fact, no collapsing is necessary for iterations of finite length - the proof of the Two Step Lemma \cite[\S4, p.~136, Theorem 1]{Jensen2014:SubcompleteAndLForcingSingapore} for subcomplete forcing goes through to show this. This would suggest that $\SCabove{\mu}$ is a very canonical class, and that the theory of its bounded forcing axioms might behave similarly to that of other forcing classes (for $\mu\neq\omega_1$). This will turn out not to be the case, however. Let me make a couple of simple observations first.

\begin{obs}
\label{obs:MeaningOfSCaboveAnOrdinal}
Let $\mu$ be an ordinal and $\P$ a notion of forcing. Then $\P$ is subcomplete above $\mu$ iff $\P$ is subcomplete above $\card{\mu}$ (the cardinality of $\mu$).
\end{obs}

\begin{proof}
For the direction from left to right, assume that $\P$ is subcomplete above $\mu$.

I will use a fact about \emph{weak} subcompleteness above $\alpha$, a concept originally introduced by Jensen in \cite[\S2, pages 3, 8]{Jensen:SPSCF} for subproperness and subcompleteness. The weakening is that in each case, it is only required that the conditions of the original definition hold if some fixed parameter $z$ is in the range of $\sigma$ (using the notation of the definition), and Jensen shows that the weak versions of these concepts are actually equivalent to the original ones. In \cite[\S3, p.~11]{Jensen:IterationTheorems}, Jensen states that the similarly defined condition of weak subcompleteness above $\alpha$ is implies (and hence is equivalent to) subcompleteness above $\alpha$.\footnote{In \cite{Jensen:IterationTheorems}, Jensen refers to subcompleteness above $\mu$ as $\mu$-subcompleteness. I will use ``subcompleteness above $\mu$,'' since this usage predates the latter one.}

To show that $\P$ is also subcomplete above $\card{\mu}$, it suffices to show that it is weakly subcomplete above $\card{\mu}$, and the parameter I want to require to be in the range of $\sigma$ is the ordinal $\mu$. The argument is then trivial, because in the notation of Definition \ref{defn:SubcompleteAboveMu}, we can ensure that $\sigma'\rest\bmu=\sigma\rest\bmu$ (which we may, because $\mu\in\ran(\sigma)$). But then $\sigma'\rest\card{\bmu}^\bN=\sigma\rest\card{\bmu}^\bN$, as required.

For the converse, assume that $\P$ is subcomplete above $\card{\mu}$. To see that $\P$ is also subcomplete above $\mu$, let $\sigma:\bN\prec N$ be as usual, with $\mu\in\ran(\sigma)$. Let $\sigma(\bmu)=\mu$, $\sigma(\bP)=\P$, and fix a $\bG\sub\bP$ generic over $\bN$.
Since $\mu\in\ran(\sigma)$, so is $\kappa=\card{\mu}$, and $\bkappa=\sigma^{-1}(\kappa)=\card{\bmu}^\bN$. Fix a surjection $\barf:\bkappa\To\bmu$, $\barf\in\bN$. By subcompleteness above $\kappa$, let $p\in\P$ be such that $p$ forces wrt.~$\P$ that there is a $\sigma':\bN\prec N$ with the usual properties, and such that $\sigma'\rest\bkappa=\sigma\rest\bkappa$. In addition, we may require that $\sigma'(\barf)=\sigma(\barf)=f$.
It follows that $\sigma'\rest\bmu=\sigma\rest\bmu$, because for $\xi<\bmu$, there is a $\zeta<\bkappa$ such that $\barf(\zeta)=\xi$, so that
\[\sigma'(\xi)=\sigma'(\barf(\zeta))=\sigma'(\barf)(\sigma'(\zeta)=\sigma(\barf)(\sigma(\zeta))=\sigma(\barf(\zeta))=\sigma(\xi).\]
This shows that $\P$ is subcomplete above $\mu$.
\end{proof}

\begin{obs}
\label{obs:SmallForcingSCabovemu}
If a forcing notion $\P$ is subcomplete above $\mu$, where $\mu\ge\card{\P}$, then $\P$ is forcing equivalent to a countably closed forcing.
\end{obs}

\begin{proof}
Let us assume that the conditions in $\P$ are ordinals less than $\mu$. In the usual setup, let $p\in\P$ be such that whenever $G\ni p$ is $\P$-generic, then in $\V[G]$, there is a $\sigma':\bN\prec N$, with $\sigma'(\va)=\sigma(\va)$, $\sigma'\rest\bmu=\sigma\rest\bmu$ (where $\sigma(\bmu)=\mu$), $\sigma'(\bP)=\sigma(\bP)=\P$, $(\sigma')``\bG\sub G$ and $\text{Hull}^N(\ran(\sigma)\cup\delta)=\text{Hull}^N(\ran(\sigma')\cup\delta)$, where $\delta=\delta(\P)$. Then since $\sigma'\rest\bmu=\sigma\rest\bmu$, it follows that $\sigma'\rest\bP=\sigma\rest\bP$, and hence, $\sigma``\bG=(\sigma')``\bG\sub G$. Thus, $\P$ is ``complete'' (which essentially means that $\sigma'$ in the definition of subcompleteness can be chosen to be equal to $\sigma$ - actually, completeness is slightly stronger than this, because in the definition of completeness, $\bN$ is not assumed to be full), and Jensen showed that complete forcing notions are forcing equivalent to countably closed ones (this argument goes through if $\bN$ is required to be full, see Minden \cite{Kaethe:Diss}).
\end{proof}

Returning to bounded forcing axioms, the following fact for the case where $\Gamma$ is the class of all proper forcing notions is due to Miyamoto \cite{Miyamoto:WeakSegmentsOfPFA}, and the version for subcomplete forcing is due to Fuchs \cite[Lemma 3.10, Lemma 4.13, Observation 4.7 and Lemma 4.9]{Fuchs:HierarchiesOfForcingAxioms}. The statement of the fact uses a large cardinal concept dating back to Miyamoto \cite{Miyamoto:WeakSegmentsOfPFA}, who introduced a hierarchy of localized reflecting cardinals, and showed among other things that $\BPFA(\omega_1,\omega_2)$ is equiconsistent with the existence of a cardinal $\kappa$ that's $\kappa^+$-reflecting (I called such a cardinal $+1$-reflecting in \cite{Fuchs:HierarchiesOfForcingAxioms}). Independently, Villaveces \cite{Villaveces:StrongUnfoldability} introduced the concept of a strongly unfoldable cardinal, which later turned out to be equivalent to a $+1$-reflecting cardinal. I showed in \cite{Fuchs:HierarchiesOfForcingAxioms} that Miyamoto's result extends to subcomplete forcing as well: the consistency strength of $\BSCFA(\omega_1,\omega_2)$ is a $+1$-reflecting cardinal. There is a distinction between unfoldability and strong unfoldability, but the consistency strengths of these concepts are the same, and in $L$ they are equivalent.

\begin{fact}
\label{fact:BFAatOmega2andUnfoldability}
Let $\Gamma$ be either the class of subcomplete or of proper forcing notions. \begin{enumerate}
  \item If $\BFA_\Gamma(\omega_1,\omega_2)$ holds, then $\omega_2$ is unfoldable in $L$.
  \item Assume $\V=L$ and $\kappa$ is unfoldable. Then there is a $\kappa$-c.c.~forcing $\P$ in $\Gamma$ such that if $G$ is generic for $\P$, then $\kappa=\omega_2^{L[G]}$ and $L[G]\models\BFA_\Gamma(\omega_1,\omega_2)$.
\end{enumerate}
\end{fact}

Part (1) of this fact goes through in the case $\Gamma=\SCabove{\omega_2}$ as well, see the proof of Lemma \ref{lem:ConsistencyStrengthOfATPatSuccessorOfContinuum}. It is less clear what would be the correct version of part (2) in this context. In the original setting, the forcing notion in part (2) can be presented as a preparatory forcing, such as a fast function forcing, which is much more than countably closed (and hence belongs to all the forcing classes of interest here), and $\kappa$-c.c., followed by a length $\kappa$ iteration of forcing notions in $\Gamma$, each iterand having size less than $\kappa$. The problem in formulating a version of part (2) for $\SCabove{\omega_2}$ is that the meaning of ``$\omega_2$'' changes throughout the iteration. A version that is true uses $\Gamma=\SCabove{\tau}$, where $\tau=\omega_2^L$. At some point in the iteration, $\omega_2^L$ is collapsed to $\omega_1$, though, and thus, from that point on, $\SCabove{\tau}$ is the same as $\SC$, by Observations \ref{obs:MeaningOfSCaboveAnOrdinal} and \ref{obs:SCaboveContinuumIsFree}. Thus, there seems to be no real difference between this approach versus just forcing with a subcomplete forcing to produce a model of $\BSCFA(\omega_1,\omega_2)$. It would be much more appealing if one could force with forcing notions that are subcomplete above $\kappa$, to produce a model in which $\kappa=\omega_2$ and $\BFA_{\SCabove{\omega_2}}(\omega_2)$ holds. However, this cannot be accomplished by iterating forcing notions of size less than $\kappa$, by Observation \ref{obs:SmallForcingSCabovemu}, since these are (equivalent to) countably closed forcing notions. Actually, it turns out that it cannot be done at all. To show this, I will use some subcomplete forcing notions.

\begin{fact}

\begin{enumerate}
  \item If $2^\omega<\tau$, where $\tau$ is a regular cardinal, and if $A\sub S^\tau_\omega$ is stationary, then the forcing notion $\P_A$ to shoot a club through $A$ is subcomplete. Conditions in $\P_A$ are increasing, continuous functions from some countable successor ordinal to $A$, ordered by reverse inclusion. See \cite[p.~134, Lemma 6.3]{Jensen2014:SubcompleteAndLForcingSingapore} for the proof, due to Jensen, which requires the extra assumption that $2^\omega<\tau$. This omission was observed by Sean Cox.
  \item Assuming \CH, Namba forcing $\naturals$, in the form where conditions are subtrees of ${}^{{<}\omega}\omega_2$ that are highly splitting, that is, a subtree $T\sub{}^{{<}\omega}\omega_2$ is in $\naturals$ iff $T\neq\leer$ and for every $s\in T$, the set $\{t\in T\st s\sub t\}$ has cardinality $\omega_2$. The ordering on $\naturals$ is inclusion. See \cite[p.~132, Lemma 6.2]{Jensen2014:SubcompleteAndLForcingSingapore} for Jensen's proof that $\naturals$ is subcomplete under \CH. Jensen proves the subcompleteness of some other Namba variants in \cite{Jensen:SCofNambaTypeForcings}.
\end{enumerate}
\end{fact}

The forcing $\P_A$ adds a subset of $A$, club in $\tau$, of order type $\omega_1$, and $\naturals$ collapses the cofinality of $\omega_2^\V$ to $\omega$, while preserving $\omega_1$ as a cardinal.

\begin{obs}
\label{obs:Impossibility}
Assume $\V=L$, and let $\kappa$ be a cardinal. There is no forcing notion $\P$ that's subcomplete above $\kappa$, such that if $G$ is $\P$-generic over $L$, then $\kappa=\omega_2^{L[G]}$ and $L[G]\models\BFA_{\SCabove{\omega_2}}(\omega_1,\omega_2)$.
\end{obs}

\begin{proof}
Assume $\P$ were such a forcing notion. I will use arguments of \cite{Fuchs:HierarchiesOfForcingAxioms} to derive a contradiction. Let $G$ be generic for $\P$. Over $L[G]$, let $H_0$ be generic for $\Col{\omega_1,\omega_2}$. Let $\B_0$ be the regular open algebra of $\Col{\omega_1,\omega_2}$, and let $\nu$ be the cardinality of this algebra.
In $L[G][H_0]$, consider Namba forcing, $\naturals$, followed by the collapse of $\nu$ to $\omega_1$. $\CH$ holds in $L[G][H_0]$, so Namba forcing is subcomplete there, and hence, so is the composition of Namba forcing with the collapse of $\nu$ to $\omega_1$. This means that it is subcomplete above $\kappa$ there (since $\kappa$ has size $\omega_1$). It follows by Theorem \ref{thm:IteratingSubcompletenessAboveMu} that $\Q=\Col{\omega_1,\omega_2}*\dot{\naturals}*\Col{\omega_1,\nu}$ is subcomplete above $\omega_2$ in $L[G]$. To see this, 
let $\mu_0=\kappa$, and let be $\B_1$ a complete Boolean algebra such that $\B_1/\dot{G}_{\B_0}$ is equivalent to a $\B_0$-name for Namba forcing followed by the collapse of $\nu$ to $\omega_1$. The conditions of Theorem \ref{thm:IteratingSubcompletenessAboveMu} are then satisfied.

Let $H_1$ be $\B_1$-generic over $L[G][H_0]$, and let $H=H_0*H_1$.

Let $\theta>\kappa$ be a regular cardinal in $L[G][H]$.
Let $\vec{C}$ be Jensen's canonical global $\square$ sequence of $L$, see \cite{FS}. Thus $\vec{C}=\seq{C_\alpha}{\alpha\ \text{is singular in $L$}}$ is $\Sigma_1$-definable in $L$ and has the property that for every $L$-singular ordinal $\alpha$, $C_\alpha$ is club in $\alpha$, has order type less than $\alpha$, and satisfies the coherence property that if $\beta<\alpha$ is a limit point of $C_\alpha$, then $\beta$ is singular in $L$ and $C_\beta=C_\alpha\cap\beta$.

Let $B=\{\xi<\theta\st\kappa<\xi<\theta\ \text{and}\ \cf(\theta)=\omega\}$. By covering, every $\xi\in B$ is singular in $L$, and hence, $C_\xi$ is defined, for every such $\xi$. Since the map $\xi\mapsto\otp(C_\xi)$ is regressive, there is a stationary $A\sub B$ on which it is constant, say with value $\beta_0$. The forcing $\P_A$ which shoots a club set of order type $\omega_1$ through $A$ is subcomplete in $L[G][H]$, and since the cardinality of $\kappa$ is $\omega_1=\omega_1^L$ in $L[G][H]$, this means that $\P_A$ is subcomplete above $\kappa$. Hence, the entire composition $\B_0*\B_1*\dot{\P_A}$ is subcomplete above $\kappa$ in $L[G]$.

Let $I$ be generic over $L[G][H]$ for $\P_A$. Then in $L[G][H][I]$, the following statement $\Phi(\omega_1)$ holds:

``there are ordinals $\alpha$, $\beta$ and $\gamma$, and a set $C$ of order type $\omega_1$, club in $\alpha$, such that for all $\xi\in C$, $C_\xi$ is defined and $\otp(C_\xi)=\beta$, $\gamma$ is an uncountable regular cardinal in $L_\alpha$, and $\gamma$ has countable cofinality.''

Since $\vec{C}$ is $\Sigma_1$-definable in $L$ (and hence in $\V$), this statement can be expressed in a $\Sigma_1$ way, using $\omega_1=\omega_1^{L[G]}$ as a parameter. Moreover, the statement holds in $L[G][H][I]$, as witnessed by $\alpha=\theta$, $\beta=\beta_0$, $\gamma=\kappa$ (which is $\omega_2^{L[G]}$), and $C$ being the club added by $I$. Since $H*I$ is generic for a forcing that's subcomplete above $\omega_2$ in $L[G]$, and since we assumed that $\BFA_{\SCabove{\omega_2}}(\omega_1,\omega_2)$ holds in $L[G]$, it follows that $\Phi(\omega_1)$ also holds in $L[G]$. Let $\alpha,\beta,\gamma,C$ witness this. Since the only parameter used in $\Phi(\omega_1)$ is $\omega_1^{L[G]}=\omega_1^L$, these witnesses may be chosen in $(H_{\omega_2})^{L[G]}$.

It now follows that $\alpha$ is a regular cardinal in $L$, since $C_\alpha$ must be undefined - otherwise, we'd have that $\otp(C_\alpha)=\omega_1$ (because $C$ is club in $\alpha$ and has order type $\omega_1$), and so there would be $\xi<\zeta$, both in $C\cap C'_\alpha$ ($C'_\alpha$ being the set of limit points less than $\alpha$ of $C_\alpha$), which would have to satisfy $C_\xi=C_\alpha\cap\xi$ and $C_\zeta=C_\alpha\cap\zeta$ by coherency, and so, $C_\xi$ would have to be a proper initial segment of $C_\zeta$, yet $\otp(C_\xi)=\otp(C_\zeta)=\beta$, a contradiction. But since $\alpha$ is a cardinal in $L$ and $L_\alpha$ believes that $\gamma$ is an uncountable regular cardinal, it follows that $\gamma$ is actually an uncountable regular cardinal in $L$. Since $\gamma$ has countable cofinality in $L[G]$, this means that $G$ added cofinal subset of $\gamma$ of order type $\omega$. But $\gamma<\omega_2^{L[G]}=\kappa$, and $G$ is generic for a forcing that's subcomplete above $\kappa$, so this is impossible.
\end{proof}

\section{Aronszajn tree preservation by subcomplete forcing}
\label{sec:ATPbySCforcing}

In the next three subsections, I will systematically analyze the property $\wATP{\omega_1,\lambda}{\SC}$ in terms of consistency strength, consequences and relationships to bounded forcing axioms. It turns out that it makes sense to break down the analysis in three subcases: $\lambda<2^\omega$, $\lambda=2^\omega$ and $\lambda>2^\omega$. The first of these cases has been solved in joint work with Minden.

\subsection{Below the continuum}
\label{subsec:BelowC}

The following preservation property of subcomplete forcing settles the matter.

\begin{thm}[Fuchs \& Minden {\cite[Theorem 4.4]{FuchsMinden:SCforcingTreesGenAbs}}]
\label{thm:SCforcingAddsNoNewBranchGeneral}
Subcomplete forcing cannot add cofinal branches to $(\omega_1,{<}2^\omega)$-trees.
\end{thm}

In particular, $\wATP{\omega_1,{<}2^\omega}{\SC}$ holds always.

\subsection{At the continuum}
\label{subsec:AtC}

Writing $\BSCFA$ for $\BFA_{\SC}$, the next lemma follows directly from Corollary \ref{cor:CharacterizationOfATP(omega1,continuum)}.

\begin{lem}
\label{lem:BSCFAatContinuum}
The following are equivalent:
\begin{enumerate}[label=(\arabic*)]
\item
\label{item:BFAatContinuum}
$\BSCFA(\omega_1,2^\omega)$,
\item
\label{item:Sigma1-1AbsolutenessAtContinuum}
$\Sigma^1_1(\omega_1,2^\omega)$-absoluteness for subcomplete forcing,
\item
\label{item:ATPatContinuum}
$\wATP{\omega_1,2^\omega}{\SC}$.
\end{enumerate}
\end{lem}

This generalizes the result \cite[Lemma 4.21]{FuchsMinden:SCforcingTreesGenAbs}, obtained jointly with Kaethe Minden, which was shown under the assumption that $2^\omega=\omega_1$.

\begin{defn}
\label{defn:SFP}
Let $\tau$ be a cardinal of uncountable cofinality. The \emph{strong Friedman property at $\tau$,} $\SFP{\tau}$, says that if $\seq{D_i}{i<\omega_1}$ is a partition of $\omega_1$ into stationary sets and $\seq{A_i}{i<\omega_1}$ is a sequence of stationary subsets of $S^{\tau}_\omega=\{\alpha<\tau\st\cf(\alpha)=\omega\}$, then there is a strictly increasing, continuous function $f:\omega_1\To\omega_2$ such that for every $i<\omega_1$, $f``D_i\sub A_i$.
\end{defn}

\begin{fact}
\label{fact:SFPandArithmetic}
Let $\tau>\omega_1$ be a regular cardinal such that $\SFP{\tau}$ holds. Then $\tau^{\omega_1}=\tau$.
\end{fact}

\begin{proof}
This is due to Foreman-Magidor-Shelah \cite{FMS:MM1}, see Jech \cite[p.~686, proof of Theorem 37.13]{ST3}.
\end{proof}

\begin{obs}
\label{obs:ATPimpliesSFP}
Let $\tau>2^\omega$ be a regular cardinal with $\tau^\omega=\tau$. Then $\wATP{\omega_1,\tau}{\SC}$ implies $\SFP{\tau}$. In fact, it suffices to have $\wATP{\omega_1,\tau}{\Gamma}$, where $\Gamma$ denotes the class of all subcomplete forcing notions that are also countably distributive.
\end{obs}

\begin{proof}
Suppose $\SFP{\tau}$ fails. Let $\vec{D}$, $\vec{A}$ as in Definition \ref{defn:SFP} be a counterexample. Let $T$ be the tree of increasing, continuous functions from some ordinal $\alpha<\omega_1$ to $\tau$ such that for all $i<\omega_1$, $f``D_i\sub A_i$. Then $T$ is a tree of height $\omega_1$ (this follows from \cite[Proof of Thm.~10, first claim, page 17]{FMS:MM1}). If $\alpha<\omega_1$, then the cardinality of the $\alpha$-th level of $T$ is at most $\tau^\alpha=\tau$. Thus, $T$ is an
$(\omega_1,\tau)$-Aronszajn tree, because a cofinal branch through it would amount to a function $g:\omega\To\tau$ that witnesses this particular instance of $\SFP{\tau}$. But a cofinal branch through $T$ can be added by forcing with a subcomplete forcing that's in fact countably distributive, see \cite[pp.~134, Lemma 6.3 and pp.~154, Lemma 7.1]{Jensen2014:SubcompleteAndLForcingSingapore} - here, the assumption that $2^\omega<\tau$, which was omitted by Jensen, seems to be needed. This contradicts $\wATP{\omega_1,\tau}{\Gamma}$.
\end{proof}

%


I would like to say a few words about the consistency strengths of the concepts under consideration. It was shown in \cite{GoldsternShelah:BPFA} that the bounded proper forcing axiom, that is, $\BFA_{\text{proper}}$, or $\BPFA$, is equiconsistent with a reflecting cardinal, a large cardinal concept introduced in that article. It was shown in \cite{Fuchs:HierarchiesOfForcingAxioms} that the consistency strength of $\BSCFA$ can be pinned down in the same way, by a reflecting cardinal.
I mentioned before that Miyamoto \cite{Miyamoto:WeakSegmentsOfPFA} showed that $\BPFA(\omega_2)$ is equiconsistent with the existence of a cardinal $\kappa$ that's $\kappa^+$-reflecting (also known as $+1$=reflecting or strongly unfoldable). I showed in \cite{Fuchs:HierarchiesOfForcingAxioms} that Miyamoto's result extends to subcomplete forcing as well: the consistency strength of $\BSCFA(\omega_2)$ is a $+1$-reflecting cardinal. Originally, the models of $\BSCFA$/$\BSCFA(\omega_2)$ produced also satisfied \CH, but modulo joint work with Corey Switzer, it is easy to see that adding $\neg\CH$ does not increase the consistency strengths. But it is currently an open question whether $\BSCFA(\omega_2)+2^\omega>\omega_2$ is consistent.


\subsection{Above the continuum}
\label{subsec:AboveC}

Let's now look at $\wATP{\omega_1,\lambda}{\SC}$ in the case that $\lambda>2^\omega$. The first possible case is that $\lambda=(2^\omega)^+$. In this case, the consistency strength of $\wATP{\omega_1,\lambda}{\SC}$ can be pinned down as follows.

\begin{lem}
\label{lem:ConsistencyStrengthOfATPatSuccessorOfContinuum}
The consistency strength of the theory $\ZFC+\wATP{\omega_1,(2^\omega)^+}{\SC}$ is a strongly unfoldable cardinal. \end{lem}


\begin{proof}
From a strongly unfoldable cardinal, one can force to a model of $\BSCFA(\omega_1,\omega_2)+\CH$ (see \cite[Lemma 4.13]{Fuchs:HierarchiesOfForcingAxioms}); the forcing described in the proof forces \CH. By Corollary \ref{cor:ATPcharacterizationTechnical}, this implies $\wATP{\omega_1,\omega_2}{\SC}$. Since \CH holds in the model produced, it satisfies $\wATP{\omega_1,(2^\omega)^+}{\SC}$.

Conversely, assume $\ZFC+\wATP{\omega_1,(2^\omega)^+}{\SC}$.
Clearly, $((2^\omega)^+)^\omega=(2^\omega)^+$.
Let $\Gamma$ be the class of subcomplete forcing notions that are countably distributive. Note that
\wATP{\omega_1,\omega_2}{\SC} implies \wATP{\omega_1,\omega_2}{\Gamma}, and $\Gamma$ is natural (this is because \SC is natural and the class of countably distributive forcing notions is closed under forcing equivalence).
Since every forcing in $\Gamma$ is $[(2^\omega)^+]^\omega$-preserving and $((2^\omega)^+)^\omega=(2^\omega)^+$, Theorem \ref{thm:CharacterizationOfATP(omega1,lambda)} applies, showing that $\BFA_\Gamma(\omega_1,(2^\omega)^+)$ holds. Letting $\kappa=(2^\omega)^+$, this implies that $\kappa$ is strongly unfoldable in $L$, and in fact, the proof of \cite[Lemma 3.10]{Fuchs:HierarchiesOfForcingAxioms} shows this. In that proof, the assumption is that $\BSCFA(\omega_1,\omega_2)$ holds, but one may replace $\omega_2$ with $(2^\omega)^+$ and run the same argument, as the two forcing notions used in the proof are subcomplete and countably distributive. So $\BFA_\Gamma(\omega_1,\kappa)$ is enough to run the proof.
\end{proof}

Let's now consider the principle $\wATP{\omega_1,\lambda}{\SC}$ for $\lambda>(2^\omega)^+$. 
Recall that by Corollary \ref{cor:ATPcharacterizationTechnical}, $\BSCFA(\omega_1,\lambda)$ implies $\wATP{\omega_1,\lambda}{\SC}$ for any cardinal $\lambda\ge\omega_1$. In particular, $\wATP{\omega_1,\lambda}{\SC}$ is consistent, assuming the consistency of the corresponding bounded forcing axiom. It will turn out that the consistency strength goes up considerably if $\lambda>(2^{\omega})^+$, though. In order to be more specific, I will digress briefly, and study the effects of $\wATP{\omega_1,\lambda}{\SC}$ on square principles. This study will be closely connected to phenomena of stationary reflection introduced in Fuchs \cite[Definition 3.2]{Fuchs:DiagonalReflection}, and will provide lower bounds on the consistency strength of Aronszajn tree preservation above the continuum.

\begin{defn}
\label{defn:DiagonalReflection}
Let $\lambda$ be a regular cardinal, let $S\sub\lambda$ be stationary, and let $\kappa<\lambda$. The \emph{diagonal reflection principle} $\DSR{{<}\kappa,S}$ says that whenever $\seq{S_{\alpha,i}}{\alpha<\lambda,i<j_\alpha}$ is a sequence of stationary subsets of $S$, where $j_\alpha<\kappa$ for every $\alpha<\lambda$, then there are a $\gamma<\lambda$ of uncountable cofinality and a club $F\sub\gamma$ such that for every $\alpha\in F$ and every $i<j_\alpha$, $S_{\alpha,i}\cap\gamma$ is stationary in $\gamma$. The version of the principle in which $j_\alpha\le\kappa$ is denoted $\DSR{\kappa,S}$.
\end{defn}

The square principles of interest are of the following kind.

\begin{defn}
\label{defn:WeakThreadedSquares}
Let $\lambda$ be a limit of limit ordinals.
A sequence $\vec{\calC}=\langle\calC_\alpha\st\alpha<\lambda,\ \alpha\ \text{limit}\rangle$ is \emph{coherent} if for every limit $\alpha<\lambda$, $\calC_\alpha\neq\leer$ and for every $C\in\calC_\alpha$, $C$ is club in $\alpha$, and for every limit point $\beta$ of $C$, $C\cap\beta\in\calC_\beta$. A \emph{thread} through $\vec{\calC}$ is a club subset $T$ of $\lambda$ that coheres with $\vec{\calC}$, that is, for every limit point $\beta$ of $T$ with $\beta<\lambda$, it follows that $T\cap\beta\in\calC_\beta$. If every $\calC_\alpha$ has size less than $\kappa$, then $\vec{\calC}$ is said to have \emph{width} ${<}\kappa$. The \emph{length} of $\vec{\calC}$ is $\lambda$.

If $\kappa$ is a cardinal, $\vec{\calC}$ has width ${<}\kappa$, and $\vec{\calC}$ does not have a thread, then $\vec{\calC}$ is called a $\square(\lambda,{<}\kappa)$ sequence. The principle $\square(\lambda,{<}\kappa)$ says that there is a $\square(\lambda,{<}\kappa)$ sequence.

In place of $\square(\lambda,{<}\kappa^+)$, I will usually write $\square(\lambda,\kappa)$.
\end{defn}

For cardinals $\kappa<\lambda$, $S^\lambda_\kappa$ denotes the set of ordinals less than $\lambda$ that have cofinality $\kappa$. The connection between diagonal stationary reflection and these square principles is as follows.

\begin{thm}[{\cite[Theorem 3.4]{Fuchs:DiagonalReflection}}]
\label{thm:DSRimpliesFailureOfSquare}
Let $\lambda$ be regular, $\kappa<\lambda$ a cardinal, and assume that $\DSR{{<}\kappa,S}$ holds, for some stationary $S\sub\lambda$. Then $\square(\lambda,{<}\kappa)$ fails.
\end{thm}

In fact, it was shown in \cite{Fuchs-LambieHanson:SeparatingDSR} that this theorem remains true if only the version of the diagonal reflection principle is assumed in which $F$ in Definition \ref{defn:DiagonalReflection} is only required to be stationary rather than closed and unbounded.
The connection between Aronszajn tree preservation and diagonal stationary reflection is:

\begin{lem}
\label{lem:wATPimpliesDSR}
If $2^\omega<\lambda=\lambda^\omega$ is regular and $\wATP{\omega_1,\lambda}{\SC}$ holds, then $\DSR{\omega_1,S^\lambda_\omega}$ holds.
\end{lem}

\begin{proof}
Let $S_{\alpha,i}\sub S^\lambda_\omega$ be stationary, for $\alpha<\mu$ and $i<\omega_1$. Let $c:\lambda\To\lambda\times\omega_1$ be a bijection, and let $T_\alpha=S_{c(\alpha)}$, for $\alpha<\lambda$. Fix a partition $\seq{A_i}{i<\omega_1}$ of $\omega_1$ into stationary sets. Let $\P=\P_{\vA,\vT}$ be the forcing described in \cite[Definition 2.23]{Fuchs:HierarchiesOfForcingAxioms}. It is shown in \cite[Lemma 2.24]{Fuchs:HierarchiesOfForcingAxioms} that $\P$ is subcomplete. The proof needs the additional assumption that $\lambda>2^\omega$, which we made here. In \cite{Larson:SeparatingSRP}, this forcing is presented as a composition of $\Col{\omega_1,\lambda}$ and a forcing to shoot a club through some stationary subset of $\omega_1$. Since both of these are countably distributive, so is the composition, $\P$. Let $\Gamma$ be the class of forcing notions that are both subcomplete and countably distributive. We then have $\wATP{\omega_1,\lambda}{\Gamma}$, which implies by Theorem \ref{thm:CharacterizationOfATP(omega1,lambda)} that $\BFA_\Gamma(\omega_1,\lambda)$ holds. Now \cite[Theorem 4.8]{Fuchs:DiagonalReflection} shows that $\BSCFA(\lambda)$ implies the existence of a simultaneous reflection point for the instance of $\DSR{\omega_1,\lambda}$ given by $\vec{S}$, using the fact that $\P$ is subcomplete. But since that forcing is also countably distributive, it clearly is enough to have $\BFA_\Gamma(\omega_1,\lambda)$ to draw this conclusion, and we are done.
\end{proof}

Putting these previous two results together, one obtains:

\begin{lem}
\label{lem:FromATPtoFailureOfSquare}
Suppose $2^\omega<\lambda=\lambda^\omega$, and \wATP{\omega_1,\lambda}{\SC} holds.
\begin{enumerate}
  \item If $\lambda=\omega_2$, then $\square(\omega_2,\omega)$ fails,
  \item if $\lambda>\omega_2$, then $\square(\lambda,\omega_1)$ fails.
\end{enumerate}
\end{lem}

Note that $2^\omega\le\omega_2+\wATP{\omega_1,\omega_2}{\SC}+\square(\omega_2,\omega_1)$ is consistent if $\wATP{\omega_1,\omega_2}{\SC}$ is, because starting in a model of $\wATP{\omega_1,\omega_2}{\SC}$, it follows that $\omega_2$ is unfoldable in $L$, and one can force over $L$ to produce a model of $\wATP{\omega_1,\omega_2}{\SC}+\CH$, which implies the weak square $\square^*_{\omega_1}$, which implies $\square(\omega_2,\omega_1)$.

We can now say something about the consistency strength of Aronszajn tree preservation by subcomplete forcing higher above $2^\omega$.

\begin{lem}
$\wATP{\omega_1,(2^\omega)^{++}}{\SC}$ implies $\AD^{L(\reals)}$.
\end{lem}

\begin{proof}
Since $((2^\omega)^+)^\omega=(2^\omega)^+$ and $((2^\omega)^{++})^\omega=(2^\omega)^{++}$,
Lemma \ref{lem:FromATPtoFailureOfSquare} applies, showing that both $\square((2^\omega)^+,\omega)$ and $\square((2^\omega)^{++},\omega_1)$ fail. Now  \cite[Theorem 5.6]{Schimmerling:CoherentSequencesAndThreads} states that if $\kappa\ge\max(2^\omega,\omega_2)$ is a regular cardinal such that $\square_\kappa$ and $\square(\kappa)$ both fail, then $M_n^\#(X)$ exists for all bounded subsets $X$ of $\kappa^+$ and for all $n<\omega$, and moreover, this conclusion was improved by Steel to $\AD^{L(\reals)}$. If we let $\kappa=(2^\omega)^+$, then the assumptions of this theorem are satisfied: since $\square(\kappa,\omega)$ fails, so does $\square(\kappa)$. And since $\square(\kappa^+,\omega_1)$ fails, so does $\square_\kappa$.
\end{proof}

Clearly, the assumption $\wATP{\omega_1,(2^\omega)^{++}}{\SC}$ in the previous lemma can be weakened to $\wATP{\omega_1,(2^\omega)^{++}}{\Gamma}$, where $\Gamma$ is the class of all subcomplete and countably distributive forcing notions.

Thus, for $\lambda\ge 2^\omega$, $\wATP{\omega_1,\lambda}{\SC}$ and $\BSCFA(\omega_1,\lambda)$ behave very similarly in terms of consistency strength, and for $\lambda=2^\omega$, they are actually equivalent. It is thus natural to ask whether they can be separated for $\lambda\neq 2^\omega$. Clearly, this is possible if $\lambda<2^\omega$, because then $\wATP{\omega_1,\lambda}{\SC}$ is a \ZFC fact, while $\BSCFA(\omega_1,\lambda)$ is not. The question is what happens if $\lambda>2^\omega$. Let us again focus on the case that $\lambda=\omega_2>2^\omega=\omega_1$. Recall that in this case, we know that $\wATP{\omega_1,\lambda}{\SCabove{\lambda}}$ is equivalent to $\BFA_{\SCabove{\lambda}}(\omega_1,\lambda)$. Here is a related question.

\begin{question}
Under \CH, is $\BSCFA(\omega_1,\omega_2)$ equivalent to $\BFA_{\SCabove{\omega_2}}(\omega_1,\omega_2)$?
\end{question}

Observation \ref{obs:Impossibility} is relevant here. In fact, that observation resulted from a failed attempt to prove that the answer is negative.
The main question that remains concerns the relationship between Aronszajn tree preservation and the bounded forcing axiom for subcomplete forcing above the continuum:

\begin{question}
Suppose $\lambda>2^\omega$ is regular. Does \wATP{\omega_1,\lambda}{\SC} imply $\BSCFA(\omega_1,\lambda)$?
\end{question}

The case $\lambda=\omega_2>2^\omega=\omega_1$ is of particular interest here. Note that in this case, there is a countably closed forcing of size $\omega_2$ that preserves $\omega_2$ and adds a $\square_{\omega_1}$-sequence, and thus destroys \wATP{\omega_1,\omega_2}{\SC}, by Lemma \ref{lem:FromATPtoFailureOfSquare}. Finally, a fundamental question is as follows:

\begin{question}
Suppose $\BSCFA(\omega_1,\omega_1)$ or $\wATP{\omega_1,2^\omega}{\SC}$ holds. Is it possible that $2^\omega>\omega_2$?
\end{question}


\end{document}